\documentclass[a4paper,12pt]{amsart}
\usepackage[ansinew]{inputenc}
\usepackage[english]{babel}
\usepackage[T1]{fontenc}
\usepackage{csquotes}
\usepackage{amssymb,amsthm,mathtools,bbm,ifxetex,mathrsfs}
\usepackage{trfsigns}
\usepackage{xspace}
\usepackage{geometry}
\usepackage{enumerate}
\usepackage{hyperref}
\usepackage{todonotes}
\numberwithin{equation}{section}
\allowdisplaybreaks

\newtheorem{satz}{Theorem}[section]
\newtheorem{proposition}[satz]{Proposition}
\newtheorem{lemma}[satz]{Lemma}
\newtheorem{korollar}[satz]{Corollary}

\newtheorem{assumptions}[satz]{Assumptions}

\newtheorem*{question*}{Question}
\newtheorem*{questions*}{Questions}
\newtheorem{bemerkung}[satz]{Remark}
\newtheorem*{bemerkung*}{Remark}

\theoremstyle{remark}

\newtheorem*{beispiel*}{Example}

\DeclarePairedDelimiter	{\abs}		{\lvert}	{\rvert}
\DeclarePairedDelimiter	{\norm}		{\lVert}	{\rVert}
\DeclarePairedDelimiter	{\skal}		{\langle}	{\rangle}

\DeclareMathOperator	{\IE}		{\mathbb{E}}
\DeclareMathOperator	{\IP}		{\mathbb{P}}
\DeclareMathOperator	{\IN}		{\mathbb{N}}

\DeclareMathOperator	{\Id}		{Id}
\newcommand				{\dpartial}	{\mathfrak{d}}
\newcommand				{\bnorm}[1]	{\bigg\lVert #1 \bigg\rVert}

\renewcommand{\epsilon}{\varepsilon}

\def\E{{\mathbb E}}

\def\P{{\mathbb P}}
\def\R{{\mathbb R}}

\def\N{{\mathbb N}}

\def\S{{\mathbb S}}

\def\XX{\mathbb{X}}
\def\Var{{\mathrm{Var}}}
\def\Ent{{\mathrm{Ent}}}
\def\vol{{\mathrm{vol}}}

\def\cs{\circ}
\def\sign{{\mathrm{sign}}}

\def\LSq{{\mathrm{LS}_q}}
\def\pist{\pi_\mathrm{St}}
\def\pis{\pi_\mathrm{sym}}
\def\rank{\mathrm{rank}}
\def\Rns{\mathbb{R}^{n \times n}_\mathrm{sym}}

\newcommand{\PP}{\mathbb{P}}

\def\cN{\mathcal{N}}

\def\sfN{{\sf N}}

\title[Higher order concentration]{Some notes on higher order concentration of measure}

\author{Holger Sambale}
\address{Holger Sambale, Faculty of Mathematics, Bielefeld University, Germany}
\email{hsambale@math.uni-bielefeld.de}

\begin{document}

\begin{abstract}
	This survey-type paper provides a common framework for a larger number of higher order concentration results (i.\,e., concentration results for non-Lipschitz functions which have bounded derivatives of higher order) in the spirit of \cite{BGS19}. Situations inlude measures satisfying various functional inequalities (log-Sobolev, Poincar\'{e}, $\mathrm{LS}_q$), uniform and cone measures on spheres with respect to the Euclidean as well as $\ell_p^n$-norms, Stiefel and Grassmann manifolds as well as discrete situations. In particular in the latter case, some open questions and remarks are stated.
\end{abstract}

\thanks{Research  supported by GRF--SFB 1283/2 2021 -- 317210226}

\date{\today}

\maketitle

\section{Introduction}

\subsection{Introductory notes}
By now a classical topic in probability theory, the origins of the concentration of measure phenomenon reach back to the 1970s. First aspects were described by V.\ D.\ Milman \cite{Mil71} in geometry and by V.\,N.\ Sudakov \cite{Sud78} in the study of the asymptotic distributional behaviour of weighted linear sums of independent or uncorrelated random variables $X_j$. In the sequel, concentration of measure has been subject to intense research with important contributions by researchers like M.\ Talagrand, M.\ Ledoux and S.\,G.\ Bobkov and others. Various overviews are available, in particular the monographs \cite{MS86,Led01,BLM13,RS14,vHa16}. One reason for the ongoing interest in concentration of measure is its richness of applications, emerging as a powerful tool in the study of high-dimensional problems in the spectral theory of random operators and random matrices as well as high-dimensional non-parametric statistics, computer science and information theory.

To sketch a typical situation, given a measure $\mu$, say, on $\R^n$, one is interested in the fluctuations of certain classes of functions around their mean value or median, hence bounding quantities like
\[
\mu\Big(\Big|f - \int f d\mu\Big| \ge t\Big)
\]
for $t \ge 0$. Here, $f$ could be a linear function $f(x) = \frac{1}{\sqrt{n}} \sum_{i=1}^n x_i$ or, more generally, a Lipschitz function. In many cases like the standard Gaussian measure, product measures with bounded support or the uniform distribution on the unit sphere, one arrives at sub-Gaussian tail bounds. Obviously, these results depend on the underlying measure, e.\,g.\ (products of) the exponential distribution will lead to subexponential tails only. In any case, there is a certain philosophy of ``many'' independent or weakly dependent observations having only limited influence on the general behaviour, so that for a vector of $n$ observations, the deviation of Lipschitz functionals from their mean is typically of order $\mathcal{O}(n^{-1/2})$.

While Lipschitz functions represent the ``classical'' situation in concentration of measure, there are many situations of interest in which the functions under consideration are not Lipschitz or have Lipschitz constants which grow as the dimension increases even after a renormalization which asymptotically stabilizes the variance. Among the simplest examples are polynomial-type functions. Here, the boundedness of the gradient has to be replaced by more elaborate conditions on higher order derivatives (up to some order $d$, which for polynomials equals their degree). Moreover, we cannot have sub-Gaussian tail decay anymore. This is already obvious if we consider the product of two independent standard normal random variables, which leads to subexponential tails. We refer to this topic as \emph{higher order concentration}.

The earliest higher order concentration inequalities date back to the late 1960s. In \cite{Bon68,Bon70,Nel73}, the growth of $L^p$ norms and hypercontractivity estimates of polynomial-type functions of Rademacher or Gaussian random variables, respectively, were studied. The question of estimating the growth of $L^p$ norms of multilinear (i.\,e., linear in every variable) polynomials in Gaussian random variables was addressed in \cite{Bor84,AG93,Lat06}. On the more applied side, in the context of Erd{\"o}s--R{\'e}nyi graphs and the triangle problem concentration inequalities for polynomial functions were studied in papers such as \cite{KV00}.

More recently, multilevel concentration inequalities have been proven in \cite{Ada06,Wol13,AW15} for many classes of functions, including $U$-statistics of independent random variables, functions of random vectors satisfying Sobolev-type inequalities and polynomials in sub-Gaussian random variables, respectively. Here, by a multilevel concentration inequality we refer to an inequality of the form
\[
\mu\Big( \Big|f - \int fd\mu\Big| \ge t \Big) \le 2 \exp\Big( - \frac{1}{C} \min_{j \in \mathcal{I}} f_j(t) \Big)
\]
for some finite index set $\mathcal{I}$. Thus, the tails expose different decay properties in certain regimes of $[0,\infty)$. In many classical situations, we have $f_j(t) = (t/C_j)^{2/j}$ for some constant $C_j$ which typically depends on the derivatives of $f$ from order $1$ to some order $d$, in which case we also speak of a $d$-th order concentration inequality.

For illustration, let us again consider the case $d=2$ and study quadratic forms of independent random variables with sub-Gaussian tail decay. In this case, the famous Hanson--Wright inequality (cf.\ \cite{HW71,Wri73}) states that if $X=(X_1, \ldots,X_n)$ is a random vector in $\R^n$ with independent components $X_i$ such that $\E X_i=0$ and $\lVert X_i \rVert_{\Psi_2} := \inf \{c > 0 \colon \E \exp(X_i^2/c^2) \le 2\} \le K$ and $A$ is an $n \times n$ matrix, then
\begin{equation}\label{eq:HWI}
\P(|X^TAX - \E X^TAX| \ge t) \le 2\exp\Big\{-c\min \Big(\Big(\frac{t}{K^2\lVert A \rVert_\mathrm{HS}}\Big)^2, \frac{t}{K^2\lVert A \rVert_\mathrm{op}}\Big)\Big\}
\end{equation}
for any $t \ge 0$. Often, one further assumes $\mathrm{Var}(X_i)=1$, in which case $\E X^TAX = \mathrm{tr}(A)$. For a modern proof, see \cite{RV13}, and for various developments, cf.\ \cite{HKZ12,VW15,Ada15,ALM18}.

The Hanson--Wright inequality hence yields that quadratic forms in independent sub-Gaussian random variables admit two levels of tail decay: for small values of $t$, a sub-Gaussian one which scales with the Hilbert--Schmidt norm of $A$, and for large values of $t$ a subexponential one which scales with the operator norm of $A$. In particular, this makes our previous observation about the tails of products of independent sub-Gaussian random variables more precise. Analogues of the Hanson-Wright inequality for sub-exponential random vectors can be found in \cite{GSS21b, Sam23}.

In some sense generalizing these observations, classical higher order concentration results like \cite{Lat06} or \cite{AW15}, continued e.\,g.\ in \cite{ABW17} or \cite{GSS21b}, yield refined (in some cases both-sided) tail bounds which are given in terms of a large family of tensor-type norms, typically for functions of polynomial type. By contrast, in \cite{BGS19} a somewhat different methodology has been developed which does not lead to the same exactness in terms of the tail bounds but turns out to provide a flexible and readily applicable framework which can be adapted to a multitude of different situations and functionals. Subsequently, the same type of arguments has been used in \cite{GSS19,GSS21a,GS23,GS24} to study various settings of higher order concentration. All these results are hence based upon a common general scheme of arguments which so far has not been formulated in full generality however.

Against this background, the principal aim of this note is to provide a framework which generalizes the stream of research building upon \cite{BGS19}. In particular, it is possible to regard the corresponding results as applications of a kind of general theorem. It turns out to be useful to distinguish between two types of situations which may roughly be labeled as continuous and discrete. Indeed, if the space $\mathcal{X}$ is discrete, a slightly more refined approach is necessary. Especially in these situations, there is also a number of open questions which we will formulate along the lines.

\subsection{Notation and conventions}

Let us fix some general notation and conventions. Our default notation for the measurable spaces involved is $(\mathcal{X}, \mathcal{B})$, typically together with a probability measure $\mu$. In many cases, $\mathcal{X}$ is equipped with a metric $d$, in which case $\mathcal{B}$ is the Borel $\sigma$-algebra. We also write $(\mathcal{X},d,\mu)$ for the respective metric probability space. Often times, the measurable space is of product structure $(\mathcal{X}, \mathcal{B}) = \otimes_{i = 1}^n (\mathcal{X}_i, \mathcal{B}_i)$, whereas $\mu$ is not necessarily a product measure.

In many situations of interest, $\mathcal{X}$ is a normed space. Typically, norms on $\mathcal{X}$ will be denoted by $\lvert \cdot \rvert_\diamondsuit$, where $\diamondsuit$ represents some specification. For instance, if $\mathcal{X} = \R^n$, we may have $\diamondsuit = p$ for some $p \ge 1$, where $\lvert \cdot \rvert_p$ denotes the $\ell^p$ norms:
\[
\lvert x \rvert_p := \begin{cases}
(\sum_{i=1}^n |x_i|^p)^{1/p}, & p \in [1, \infty)\\
\max_{i = 1,\ldots, n} |x_i|, & p = \infty.
\end{cases}
\]
For $p=2$, the notation $\lvert \cdot \rvert$ shall always indicate the Euclidean norm on $\R^n$.

Especially in the Euclidean case $\mathcal{X} = \R^n$, two types of norms for $j$-tensors $T \in \R^{n^j}$ are of particular interest: the Hilbert--Schmidt (Frobenius, Euclidean) norm $\abs{T}_\mathrm{HS} \coloneqq (\sum_{i_1, \ldots, i_j} A_{i_1 \ldots i_j}^2)^{1/2}$ and the operator norm
\begin{align*}				
	\abs{T}_\mathrm{\mathrm{op}} \coloneqq \sup_{\substack{v^1, \ldots, v^j \in \R^n \\ \abs{v^\ell} \le 1}} \skal{T, v^1 \cdots v^j} = \sup_{\substack{v^1, \ldots, v^j \\ \abs{v^\ell} \le 1}} \sum_{i_1,\ldots,i_j} T_{i_1 \ldots i_j} v^1_{i_1} \cdots v^j_{i_j},
\end{align*}
using the outer product $(v^1 \cdots v^j)_{i_1 \ldots i_j} = \prod_{\ell = 1}^j v^\ell_{i_\ell}$.

Let $f \colon \mathcal{X} \to \R$ be a measurable function on $(\mathcal{X}, \mathcal{B}, \mu)$. For any $r > 0$, we write
\[
\lVert f \rVert_r := \Big( \int \lvert f \rvert^r d\mu \Big)^{1/r}
\]
(as opposed to the $\ell^p$ norms $\lvert x \rvert_p$). Moreover, if $f$ is integrable, we write
\[
\E f = \E_\mu f = \int f d\mu.
\]

In the case of higher order derivatives, we will often take $L^r$ norms of norms $\lvert \cdot \rvert_\diamondsuit$ of some random tensor $T$. A concrete example is the tensor of $j$-th order derivatives $T = \nabla^{(j)} g$, equipped with the operator or the Hilbert--Schmidt norm. Here, we frequently use the short-hand notation
\[
\lVert T \rVert_{\diamondsuit, r} := \lVert \lvert T \rvert_\diamondsuit \rVert_r
\]
for any $r > 0$.

\subsection{Overview}

In Section \ref{sec:GenRes}, a first general result is presented and proven which especially encompasses many non-discrete situations. Corresponding examples are given in Section \ref{sec:Ex1}. As these examples are based on previous work, proofs will only be sketched in general. In Section \ref{sec:GenResDisc}, yet another general result will be stated which in particular addresses discrete-type situations. Examples are provided in Section \ref{sec:Ex2}, again only with sketches of the proofs but including some more extensive discussion.

\section{A general higher order concentration result}\label{sec:GenRes}

Our first aim is to state a result which generalizes many previously known situations, in particular measures on $\R^n$ which satisfy suitable functional (log-Sobolev, Poincar\'{e}, $\mathrm{LS}_q$) inequalities or classical probability measures on spheres and related manifolds. To this end, let us introduce the notion of \emph{difference operators}. Let $\mathcal{F}$ be some class of measurable functions $f \colon \mathcal{X} \to \R$ such that for any $a > 0$ and any $b \in \R$, we have $af+b \in \mathcal{F}$. Then, a difference operator is an operator $\Gamma \colon \mathcal{F} \to \R$ which satisfies
\begin{equation}\label{eq:DiffOp}
\Gamma(af+b) = a\Gamma(f)
\end{equation}
for any $f \in \mathcal{F}$, any $a > 0$ and any $b \in \R$. For instance, if $\mathcal{X}$ is a metric space, $\mathcal{F}$ is typically the set of all locally Lipschitz functions, and $\Gamma = |\nabla f|$ is (a generalization of) the norm of the gradient.

Two ingredients form the core of our results: first, the control of the $L^r$ norms of the functions under consideration and second, a recursion scheme which allows us to move on to higher orders. Against this background, we now state two fundamental assumptions on which our results rely.

\begin{assumptions}\label{ass:Set1}
\begin{enumerate}
\item There exist $p > 0$, $r_0 > 1$, and $L, \sigma > 0$ such that for any function $g \in \mathcal{F}$, we have
\begin{equation*}
	\lVert g - \E g \rVert_r \le L\sigma r^{1/p} \lVert \Gamma g \rVert_r
\end{equation*}
for all $r \ge r_0$.
\item There exist ``higher order'' difference operators $\Gamma^{(j)}(g) \in [0,\infty)$, $j = 1, \ldots, d$, such that for any $g \in \mathcal{F}$,
\begin{equation*}
	\Gamma (\Gamma^{(j)}(g)) \le \Gamma^{(j+1)}(g)
\end{equation*}
for any $j = 1, \ldots, d-1$.
\end{enumerate}
\end{assumptions}

A typical choice in (1) is $p=r_0=2$, where $p=2$ indicates sub-Gaussian tail decay. The reason for introducing both $L$ and $\sigma$ instead of a single factor is that in many cases, we have some ``purely numerical'' part $L$ and a factor of $\sigma$ which contains some kind of information about the measure $\mu$ (for instance, a log-Sobolev constant which may depend on the dimension). Ultimately, it is a matter of convention.

In (2), the classical case is $\Gamma^{(j)}(g) = \lvert g^{(j)} \rvert_\mathrm{op}$, the operator norm of the tensor of $j$-fold partial derivatives. More generally, we often consider a sort of tensor of higher order ``derivatives'' $\nabla^{(j)} g \in \R^{n^j}$, of which we then take some norm $\lvert \cdot \rvert_\diamondsuit$, i.\,e., $\Gamma^{(j)} g = \lvert \nabla^{(j)} g \rvert_\diamondsuit$.

Let us now formulate our first main result. Here we prefer to speak of ``suitable'' functions instead of further specifications for the sake of simplicity. In the examples section, $f$ will typically be chosen to be $\mathcal{C}^d$-smooth.

\begin{satz}\label{thm:HigherOrder}
    Given Assumptions \ref{ass:Set1}, let $f\colon \mathcal{X} \to \R$ be a suitable function with $\E f = 0$.
    \begin{enumerate}
        \item Assuming that $\lVert \Gamma^{(j)} f \rVert_1 \le \sigma^{d-j}$ for all $j=1,\ldots,d-1$ and $\lVert \Gamma^{(d)} f \rVert_\infty \le 1$, it holds that
        \[
        \int \exp\Big\{\frac{c_{p,d,r_0,L}}{\sigma^p} |f|^{p/d} \Big\} d\mu \le 2,
        \]
        where a possible choice of the constant $c_{p,d,r_0,L}$ is given by
        \[
        c_{p,d,r_0,L} = \frac{(r_0^{1/p}-1)^p}{2e\max(L^{1/d}, L)^pr_0\max(r_0,p/d)}.
        \]
        \item For any $t\ge 0$, it holds that
        \[
        \mu(|f| \ge t)\le
        2 \exp\Big\{- \frac{C_{p,r_0,L}}{d^p\sigma^p} \min\Big(\min_{j=1, \ldots,d-1} \Big(\frac{t}{\lVert \Gamma^{(j)} f \rVert_1}\Big)^{p/j}, \Big(\frac{t}{\lVert \Gamma^{(d)} f \rVert_\infty}\Big)^{p/d}\Big)\Big\},
        \]
    where a possible choice of the constant $C_{p,r_0,L}$ is given by
    \[
    C_{p,r_0,L} = \frac{\log(2)}{r_0(Le)^p}.
    \]
    \end{enumerate}
\end{satz}

Obviously, Theorem \ref{thm:HigherOrder} (1) yields weaker tail estimates than part (2). On the other hand, an interesting feature of (1) is that the constants involved only depend on $d$ in a very weak sense (for instance, not at all whenever $p/d \le r_0$ and $L \ge 1$), in contrast to the factor of $1/d^p$ in (2) and in even sharper contrast to all the results following \cite{Lat06} since the latter typically involve symmetrization and decoupling inequalities which lead to a very rough $d$-dependence.

To check higher order concentration results in concrete situations, it is natural to consider the easiest non-trivial case $d=2$ and ask for Hanson--Wright type results. Temporarily writing $f(x) := x^TAx$ as a function on the Euclidean space equipped with the usual gradient, note that $A$ is the Hessian of $f$ (up to constants). Therefore, to get back \eqref{eq:HWI} in the framework of this note we need a version of Theorem \ref{thm:HigherOrder} specialized to $d=2$ where $\lVert \Gamma^{(1)} f \rVert_1$ is replaced by a ``second order'' quantity like the Hilbert--Schmidt norm of $f''$ in the present example.

In order to avoid technicalities, we slightly specialize the situation. Taking up the notation of Assumptions \ref{ass:Set1}, we restrict ourselves to the case where $\Gamma^{(k)} g = \lvert \nabla^{(k)} g \rvert_\diamondsuit$ for some sort of (tensor of) $k$-th order derivatives $\nabla^{(k)} g$ and a suitable (typically operator-type) norm $\diamondsuit$. Moreover, we shall also consider another set of difference operators $\lvert \nabla^{(k)} g \rvert_\square$, to be thought of as Hilbert--Schmidt type norms of tensors of derivatives. In particular, $\nabla^{(1)}g$ is a vector, and we write $\E \nabla^{(1)}g$ for its componentwise integral with respect to $\mu$.

\begin{korollar}\label{cor:HWI}
Given Assumptions \ref{ass:Set1} for $\Gamma^{(k)} g = \lvert \nabla^{(k)} g \rvert_\diamondsuit$, further assume that
\begin{equation}\label{eq:ass3}
\lVert \nabla^{(1)} g \rVert_{\diamondsuit, r_0} \le L\sigma \lVert \nabla^{(2)} g \rVert_{\square, r_0}
\end{equation}
for any $g \in \mathcal{F}$ with $\E\nabla^{(1)}g = 0$ for all $i$ and the quantities $r_0$, $L$, $\sigma$ from Assumptions \ref{ass:Set1}. Then, for any $t\ge 0$ and any suitable function $f$ such that $\E f = 0$, it holds that
        \[
        \mu(|f| \ge t)\le
        2 \exp\Big\{- C_{p,r_0,L} \min\Big(\Big(\frac{t}{L\sigma^2 \lVert \nabla^{(2)} f \rVert_{\square, r_0}}\Big)^p, \Big(\frac{t}{\sigma^2\lVert \nabla^{(2)} f \rVert_{\diamondsuit,\infty}}\Big)^{p/2}\Big)\Big\},
        \]
    where a possible choice of the constant $C_{p,r_0,L}$ is given by
    \[
    C_{p,r_0,L} = \frac{\log(2)}{r_0(2Le)^p}.
    \]
\end{korollar}

Typically, the additional assumption \eqref{eq:ass3} readily follows from a Poincar\'{e}-type inequality. Note also that in principle, \eqref{eq:ass3} could be reformulated by replacing $L$ and $\sigma$ by, say, $L'$ and $\sigma'$. Corollary \ref{cor:HWI} can be easily adapted to this situation. We do not pursue this approach for clarity of presentation and in order to avoid transforming this note into a list of cases of various generality.

Let us now directly turn to the proof of Theorem \ref{thm:HigherOrder}. To this end, we need the following elementary lemma.

\begin{lemma}\label{lem:Conc}
	\begin{enumerate}
		\item Assuming that $\lVert g \rVert_m \le \gamma m$ for any $m \in \N$ and some constant $\gamma > 0$, it holds that
		\[
		\int \exp\Big\{ \frac{|g|}{2\gamma e}\Big\} d\mu \le 2.
		\]
		\item Assume that for some finite index set $\mathcal{J}$ and every $j \in \mathcal{J}$ there are constants $C_j > 0$ and real numbers $p_j > 0$, $r_0 \ge 1$ such that for any $r \ge r_0$
		\[
		\lVert g \rVert_r \le \sum_{j \in \mathcal{J}} (C_j r)^{1/p_j}.
		\]
		Then, we have
		\[
		\mu(|g| \ge t) \le 2 \exp\Big\{-\frac{\log(2)}{r_0(|\mathcal{J}|e)^{\max_j p_j}} \min_{j \in \mathcal{J}} \frac{t^{p_j}}{C_j}\Big\}
		\]
		for any $t \ge 0$.
	\end{enumerate}
\end{lemma}

\begin{proof}
	To see (1), note that  setting $c = 1/(2 \gamma e)$ and using $m! \ge (\frac{m}{e})^m$, we have
	\[
	\int \exp(c|g|) d\mu = 1 + \sum_{m=1}^{\infty} c^m \frac{\int |g|^m d\mu}{m!} 
	\le 1 + \sum_{m=1}^{\infty} (c \gamma)^m \frac{m^m}{m!}
	\le 1 + \sum_{m=1}^{\infty} (c \gamma e)^m = 2.
	\]
	
	To see (2), first note that by Chebyshev's inequality we have for any $r \ge 1$
	\begin{equation}  \label{eq:ChebyshevWithLp}
		\mu (|g| \ge e \lVert g \rVert_r) \le \exp(-r).
	\end{equation}
	Now consider the function
	\[
	\eta_g(t) := \min_{j\in\mathcal{J}} \frac{t^{p_j}}{C_j}.
	\]
	If we assume $\eta_g(t) \ge r_0$, we can estimate $e \lVert g \rVert_{\eta_g(t)} \le e \sum_{j \in \mathcal{J}} t = |\mathcal{J}|et$, so that an application of equation \eqref{eq:ChebyshevWithLp} to $r = \eta_g(t)$ yields
	\[
	\mu(|g| \ge |\mathcal{J}|et) \le \mu(|g| \ge e\lVert g \rVert_{\eta_g(t)}) \le \exp\left( - \eta_g(t) \right).
	\]
	Combining it with the trivial estimate $\mu(\cdot) \le 1$ (in the case $\eta_f(t) \le r_0$) gives
	\[
	\mu(|g| \ge |\mathcal{J}|et) \le e^{r_0} \exp(-\eta_g(t)).
	\]
	Replacing $t$ by $t/(|\mathcal{J}|e)$ we obtain
	\[
	\mu(|g| \ge t) \le e^{r_0} \exp \left( - \frac{1}{(|\mathcal{J}|e)^{\max_jp_j}} \eta_g(t) \right).
	\]
	
	Now recall that for any two constants $C_1 > C_2 > 1$ we have for all $r \ge 0$ and $C > 0$
	\begin{equation}\label{eqn:constantAdjustment} 
	C_1 \exp(-r/C) \le C_2 \exp\Big(-\frac{\log(C_2)}{C\log(C_1)} r\Big)
	\end{equation}
	whenever the left hand side is smaller or equal to $1$ (cf.\ e.\,g.\ \cite[Eq.\ (3.1)]{SS21}). In particular, the claim now follows by applying \eqref{eqn:constantAdjustment} with $C_1 = e^{r_0}$ and $C_2 = 2$.
\end{proof}

\begin{proof}[Proof of Theorem \ref{thm:HigherOrder}]
	For any $j = 1, \ldots, d-1$, setting $g :=  \Gamma^{(k)} f$ and combining Assumptions \ref{ass:Set1} (1)\&(2) yields
	\[
	\lVert \Gamma^{(j)} f - \E\Gamma^{(j)} f \rVert_r \le L\sigma r^{1/p} \lVert \Gamma^{(j+1)} f \rVert_r
	\]
	and hence, by triangle inequality,
	\[
	\lVert \Gamma^{(j)} f \rVert_r \le \lVert \Gamma^{(j)} f \rVert_1 + L\sigma r^{1/p} \lVert \Gamma^{(j+1)} f \rVert_r
	\]
	for any $r \ge r_0$. This yields
	\begin{equation}\label{eq:iteratedBd}
		\lVert f \rVert_r \le \sum_{j=1}^{d-1} L^j\sigma^j r^{j/p} \lVert \Gamma^{(j)} f \rVert_1 + L^d\sigma^d r^{d/p} \lVert \Gamma^{(d)} f \rVert_\infty
	\end{equation}
	for any $r \ge r_0$.
	
	To see (1), we assume $L \ge 1$ in the sequel (the case $L < 1$ follows in the same way by replacing $L^d$ by $L$). Plugging the assumptions into \eqref{eq:iteratedBd} leads to
	    \begin{align*}
	    \lVert f \rVert_r &\le \sigma^d L^d \sum_{j=1}^d r^{j/p}
	    \le \sigma^d L^d \frac{1}{1-r^{-1/p}} r^{d/p}\\
	    &\le \sigma^d L^d \frac{r_0^{1/p}}{r_0^{1/p}-1} r^{d/p}
	    = \Big(\sigma^p L^p \frac{r_0}{(r_0^{1/p}-1)^p} r\Big)^{d/p}
	    \end{align*}
	    for any $r \ge r_0$. If $r \le r_0$, we obtain
	    \begin{equation*}
	    \lVert f \rVert_r \le \lVert f \rVert_{r_0} \le \Big(\sigma^p L^p \frac{r_0}{(r_0^{1/p}-1)^p} r_0\Big)^{d/p}.
	    \end{equation*}
	    Hence, for any $m\ge 1$ it holds that
	    \[
	    \lVert |f|^{p/d} \rVert_m = \lVert f \rVert_{pm/d}^{p/d} \le \gamma m,\qquad \gamma = \sigma^p L^p \frac{r_0}{(r_0^{1/p}-1)^p} \max\Big(r_0, \frac{p}{d}\Big).
	    \]
	    Now it follows by Lemma \ref{lem:Conc} (1) applied to $g=|f|^{p/d}$ that $\int \exp(c'|f|^{p/d}) d\mu \le 2$ for $c'=1/(2\gamma e)$, which proves (1).
	
	    To show (2), we apply Lemma \ref{lem:Conc} (2) with $\mathcal{J} = \{1, \ldots, d\}$ and $p_j=p/j$ to \eqref{eq:iteratedBd}. This leads to a bound of $\mu(|f| \ge t)$ by
	    \[
	    2 \exp\Big\{- \frac{\log(2)C'}{r_0d^pe^p} \min\Big(\min_{k=1, \ldots,d-1} \Big(\frac{t}{\lVert \Gamma^{(k)} f \rVert_1}\Big)^{p/k}, \Big(\frac{t}{\lVert \Gamma^{(d)} f \rVert_\infty}\Big)^{p/d}\Big)\Big\},
	    \]
	    where
	    \[
	    C'= \frac{1}{L^p\sigma^p}.
	    \]
	    This finishes the proof.
\end{proof}

\section{Examples}\label{sec:Ex1}

In this section, we provide numerous examples where Theorem \ref{thm:HigherOrder} applies. Most of them involve the following generalization of the norm of the gradient. Let $(\mathcal{X},d)$ be a metric space, and let $f \colon \mathcal{X} \to \R$ be locally Lipschitz. Then, for any $x \in \mathcal{X}$ which is no isolated point,
\begin{equation}
	\label{eq:genmod}
	\lvert \nabla^\ast f(x) \rvert := \limsup_{d(x,y) \to 0^+} \frac{|f(x)-f(y)|}{d(x,y)}
\end{equation}
is called the generalized modulus of the gradient. If $x\in\mathcal{X}$ is an isolated point, we formally set $\lvert \nabla^\ast f(x) \rvert:=0$. The generalized modulus of the gradient preserves many identities from calculus in form of inequalities, such as a ``chain rule inequality''
\begin{equation}
	\label{eq:chainrule}
	\lvert \nabla^\ast (u(f)) \rvert \le |u'(f)|\lvert \nabla^\ast(f) \rvert,
\end{equation}
where $u$ is a smooth function on $\R$ (or, more generally, a locally Lipschitz function).

For metric probability spaces, \eqref{eq:genmod} is typically the a priori choice for the difference operator $\Gamma$. In particular, the standard example for \eqref{eq:genmod} is $\mathcal{X}=\R^n$ together with the Euclidean norm $|\cdot|$. In this case, if $f \colon \R^n \to \R$ is a differentiable function, $\lvert \nabla^\ast f \rvert = |\nabla f|$ is the Euclidean norm of the usual gradient. We therefore tend to write $|\nabla f|$ instead of $|\nabla^\ast f|$ permanently, keeping in mind that in the case of differentiability both notions agree and if not, the formal definition \eqref{eq:genmod} has to be used.

Similar observations hold for the more general case where $\mathcal{X} \equiv M$ is an embedded (Riemannian) submanifold of the Euclidean space $\R^n$. In this situation, one may introduce an intrinsic notion of derivatives as follows. We call a function $f \colon M \to \R$ differentiable if it can be extended to a differentiable function on some neighbourhood of $M$, i.\,e., there is some open set $U \subset \R^n$ and a differentiable function on $U$ which agrees with $f$ on $M$. For simplicity, we will denote the differentiable extension of $f$ by $f$ as well.

Denote by $T_xM$ the tangent space in $x \in M$, and let $P_x \colon \R^n \to T_xM$ the respective orthogonal projection. Then, the intrinsic gradient of $f$ at $x \in M$ is given by
\begin{equation}\label{eq:intgrad}
	\nabla_M f(x) := P_x \nabla f(x),
\end{equation}
where $\nabla f(x)$ is the usual Euclidean gradient of $f$ at $x$. This definition does not depend on the choice of the differentiable extension of $f$: indeed, if $g$ is another differentiable extension onto some neighbourhood of $M$, then $\nabla_M f(x) = \nabla_M g(x)$ for all $x \in M$. If $d$ is the Riemannian metric on $M$, then $\lvert \nabla_M f(x) \rvert$ equals the generalized modulus of the gradient of $f$ at $x \in M$ as defined in \eqref{eq:genmod}. In particular, we therefore mostly write $|\nabla_M f(x)|$ instead of $|\nabla^\ast f(x)|$ similarly as in the Euclidean case.

Note that by a different approach, we could also exclusively work on $M$ and avoid any extensions of the function $f$ to some neighbourhood. Here, we call a function $f \colon M \to \mathbb{R}$ differentiable at $x \in M$ if it admits a Taylor expansion
\begin{equation}\label{eq:defintrder}
	f(x') = f(x) + \langle v,x' - x\rangle + o\big(|x' - x|\big) \quad {\rm as} \ \ x' \rightarrow x, \ \ x' \in M
\end{equation}
with some $v \in \R^n$. The unique $v_0$ of smallest (Euclidean) length among all such $v$ is called the intrinsic derivative or gradient of $f$ at $x$, denoted $\nabla_M f(x)$.

To see that this definition of $\nabla_M f$ agrees with \eqref{eq:intgrad}, note that if $f$ is defined and smooth in an open neighbourhood of $M$, the Euclidean gradient $\nabla f(x)$ clearly satisfies \eqref{eq:defintrder}, and we may project it onto $T_x$, hence minimizing its length, which gives back our initial definition. Finding suitable expansions of the function $f$ is typically no problem, and since \eqref{eq:intgrad} is typically easier to deal with in concrete situations, we mostly work with the first approach.

\subsection{Logarithmic Sobolev inequalities}\label{sec:LSIs}

Perhaps the most classical situation in which Theorem \ref{thm:HigherOrder} applies is the case where $\mathcal{X} = \R^n$ is equipped with the Euclidean norm and $\mu$ satisfies a logarithmic Sobolev inequality with respect to the usual gradient (more precisely, the generalized norm of the gradient).

Let us introduce the notion of logarithmic Sobolev inequalities in a more general setting. We say that a probability measure $\mu$ on $(\mathcal{X},\mathcal{B})$ satisfies a logarithmic Sobolev inequality (log-Sobolev inequality, LSI) with respect to some difference operator $\Gamma$ if there is some $\sigma^2 > 0$ such that for every $f \in \mathcal{F}$ (where $\mathcal{F}$ is typically the class of all locally Lipschitz functions),
\begin{equation}\label{eq:genLSI}
	\mathrm{Ent}_\mu(f^2) \le 2\sigma^2 \int (\Gamma f)^2 d\mu,
\end{equation}
where $\mathrm{Ent}_\mu(f^2) := \int f^2\log(f^2) d\mu - (\int f^2 d\mu) \log(\int f^2 d\mu)$ denotes the entropy of $f^2$. A log-Sobolev inequality implies a Poincar\'{e}-type inequality with the same constant, i.\,e.,
\begin{equation}\label{eq:genPI}
	\mathrm{Var}_\mu(f) := \int f^2 d\mu - \int f^2 d\mu \le \sigma^2 \int (\Gamma f)^2 d\mu
\end{equation}
for every $f \in \mathcal{F}$.

In the present section, we hence assume that there exists $\sigma^2 > 0$ such that for every sufficiently smooth function $f \colon \R^n \to \R$,
\begin{equation}\label{eq:LSI}
	\mathrm{Ent}_\mu(f^2) \le 2\sigma^2 \int |\nabla f|^2 d\mu.
\end{equation}
The most classical example of a measure satisfying a log-Sobolev inequality is the standard Gaussian measure (on $\R^n$), where $\sigma^2=1$ (cf.\ \cite{Gro75}). Necessary and sufficient conditions for probability measures on $\R$ to satisfy a log-Sobolev inequality have been found in \cite{BG99}. Products of probability measures satisfying a log-Sobolev inequality (with constant $\sigma^2$) also satisfy a log-Sobolev inequality (with constant $\sigma^2$). Further examples (typically for measures on manifolds) will follow in the course of this section. For an overview, also cf.\ \cite[Ch.\ 7]{BCG23}.

It is standard knowledge that measures which satisfy a log-Sobolev inequality admit sub-Gaussian concentration for Lipschitz functions, i.\,e., for any $1$-Lipschitz function $f \colon \R^n \to \R$, it holds that
\[
\mu(|f-\E f| \ge t) \le 2\exp\Big(-\frac{t^2}{2\sigma^2}\Big)
\]
for every $t \ge 0$. By verifying Assumptions \ref{ass:Set1}, we may extend this result to higher orders, which leads to the following result. Here, given a $\mathcal{C}^d$-function $f \colon \R^n \to \R$, we denote its tensor of $j$-fold derivatives by $f^{(j)}$. In other words, its entries
\begin{equation*}
	f^{(j)}_{i_1 \ldots i_j}(x) = \partial_{i_1 \ldots i_j} f(x)
\end{equation*}
are given by the $j$-fold (continuous) partial derivatives of $f$ at $x \in \R^n$.

\begin{satz}\label{thm:HigherOrderLSI}
	Let $\mu$ be a probability measure on $\R^n$ which satisfies a log-Sobolev inequality \eqref{eq:LSI} with constant $\sigma^2 > 0$, and let $f\colon \R^n \to \R$ be a $\mathcal{C}^d$-smooth function with $\E f = 0$.
	\begin{enumerate}
		\item Assuming that $\lVert f^{(j)} \rVert_{\mathrm{op},1} \le \sigma^{d-j}$ for all $j=1,\ldots,d-1$ and $\lVert f^{(d)} \rVert_{\mathrm{op},\infty} \le 1$, it holds that
		\[
		\int \exp\Big\{\frac{c}{\sigma^2} |f|^{2/d} \Big\} d\mu \le 2,
		\]
		where we may choose $c=(\sqrt{2}-1)^2/(8e)$.
		\item For any $t\ge 0$, it holds that
		\[
		\mu(|f| \ge t)\le
		2 \exp\Big\{- \frac{C}{(d\sigma)^2} \min\Big(\min_{j=1, \ldots,d-1} \Big(\frac{t}{\lVert f^{(j)} \rVert_{\mathrm{op},1}}\Big)^{2/j}, \Big(\frac{t}{\lVert f^{(d)} \rVert_{\mathrm{op},\infty}}\Big)^{2/d}\Big)\Big\},
		\]
		where we may choose $C = \log(2)/(2e^2)$.
	\end{enumerate}
\end{satz}

Theorem \ref{thm:HigherOrderLSI} gives back Th.\ 1.9 and Cor.\ 1.11 from \cite{BGS19} up to constants and with the (minor) difference of involving $L^1$ norms $\lVert f^{(j)} \rVert_{\mathrm{op},1}$ instead of $L^2$ norms. The latter is a consequence of a slightly refined proof strategy which leads to a somewhat sharpened result, though in applications one may prefer to work with the sometimes more easily computable $L^2$ norms after all. More intricate bounds involving a family of tensor-type norms can be found in \cite[Th.\ 1.2]{AW15}.

By the same strategy, we also obtain a Hanson--Wright inequality in presence of a log-Sobolev inequality.

\begin{proposition}\label{prop:HWILSU}
    Let $\mu$ be a probability measure on $\R^n$ which satisfies a log-Sobolev inequality as well as $\int x_i d\mu = 0$ for any $i=1, \ldots,n$, and define $f\colon \R^n \to \R$ by $f(x) := x^TAx$, where $A=(a_{ij})_{ij}$ is a symmetric $n \times n$ matrix. Then,
    \[
    \mu(|f - \E f| \ge t) \le 2\exp\Big\{-c \min \Big(\Big(\frac{t}{\sigma^2\lvert A \rvert_{\mathrm{HS}}}\Big)^2, \frac{t}{\sigma^2\lvert A \rvert_{\mathrm{op}}}\Big)\Big\}
    \]
    for any $t \ge 0$.
\end{proposition}

Assuming $A$ to be symmetric is a matter of convenience, and the result can easily be extended to non-symmetric matrices $A$. Noting that the sub-Gaussian norms of $x_i$ under $\mu$ are of order $\sigma$, Proposition \ref{prop:HWILSU} can thus be regarded as a Hanson--Wright inequality in presence of a log-Sobolev inequality. Of course, this excludes every situation where the components of $X$ are independent but do not satisfy a log-Sobolev inequality, but on the other hand Proposition \ref{prop:HWILSU} also allows for dependencies between the coordinates, which is not part of the original Hanson--Wright inequality.

To show Theorem \ref{thm:HigherOrderLSI}, we check Assumptions \ref{ass:Set1}. For (1), recall that a classical result by Aida and Stroock \cite{AS94} yields
\begin{equation}\label{eq:AS}
	\lVert g \rVert_r^2 \le \lVert g \rVert_2^2 + \sigma^2(r-2) \lVert \nabla g \rVert_r^2
\end{equation}
for any $r \ge 2$. Mostly for the sake of later comparison, let us briefly recall its proof.

\begin{proof}[Proof of \eqref{eq:AS}]
We may assume $g$ to be bounded, in which case the squares of
the $L^r(\mu)$-norms of $g$ have derivatives
\begin{equation}\label{eq:Abl}
	\frac{d}{dr}\, \lVert g \rVert_r^2 = 
	\frac{2}{r^2}\, \lVert g \rVert_r^{2-r}\, \mathrm{Ent}_\mu(|g|^r).
\end{equation}
We apply this identity to the function $u = |g|^{r/2}$. By the chain rule inequality  \eqref{eq:chainrule}, $|\nabla u|^2 \le \frac{r^2}{4}\, |g|^{r-2}\, |\nabla g|^2$.
Hence, by H\"older's inequality,
$$
\int |\nabla u|^2\, d\mu \le 
\frac{r^2}{4} \Big(\int|g|^r\,d\mu\Big)^{\frac{r-2}{r}}
\Big(\int |\nabla g|^r\,d\mu\Big)^{\frac{2}{r}} = 
\frac{r^2}{4} \lVert g \rVert_r^{r-2}\lVert \nabla g \rVert_r^2.
$$
Applying \eqref{eq:LSI} to the function $u$, we therefore obtain
$$
\mathrm{Ent}_\mu (|g|^r) = \text{Ent}_\mu (u^2) \le 
2 \sigma^2 \int |\nabla u|^2\, d\mu \le 
\frac{r^2 \sigma^2}{2}\, \lVert g \rVert_r^{r-2}\, \lVert \nabla g \rVert_r^2.
$$
Combining this with \eqref{eq:Abl}, we arrive at the differential inequality 
$\frac{d}{dr}\, \lVert g \rVert_r^2 \le \sigma^2 \lVert \nabla g \rVert_r^2$. 
Integrating it from $2$ to $r$ yields \eqref{eq:AS}.
\end{proof}

By Poincar\'{e} inequality, \eqref{eq:AS} implies
\begin{equation}\label{eq:1fLSI}
	\lVert g - \E g \rVert_r^2 \le \sigma^2(r-1) \lVert \nabla g \rVert_r^2
\end{equation}
In particular, Assumption \ref{ass:Set1} (1) is satisfied with $p=r_0=2$ and $L=1$.

For Assumption \ref{ass:Set1} (2), the higher order difference operators are given by $\Gamma^{(j)} g(x) := |g^{(j)}(x)|_{\mathrm{op}}$. In particular, $\Gamma^{(1)} g = |g^{(1)}|_{\mathrm{op}} = |\nabla g|$ is the Euclidean norm of the gradient. Now, the following lemma holds, which establishes Assumption \ref{ass:Set1} (2).

\begin{lemma}[\cite{BGS19}, Le.\ 4.1]
	\label{lem:itGradHessLSI}
	Given a $\mathcal{C}^{j+1}$-smooth function $f \colon \R^n \to \mathbb{R}$, $j \in \mathbb{N}$, for any $x \in \R^n$ it holds that
	\[
	|\nabla \lvert f^{(j)}(x) \rvert_{\mathrm{op}}| \le 
	\lvert f^{(j+1)}(x) \rvert_{\mathrm{op}}.
	\]
\end{lemma}

\begin{proof} For any $h \in \mathbb{R}^n$, by the triangle inequality,
	\begin{align*}
		& 
		\big|\,\lvert f^{(j)}(x+h) \rvert_{\mathrm{op}} - 
		\lvert f^{(j)}(x) \rvert_{\mathrm{op}}\big|
		\, \le \,
		\lvert f^{(j)}(x+h) - f^{(j)}(x) \rvert_{\mathrm{op}}\\
		= \, &
		\sup \{ (f^{(j)}(x+h) - f^{(j)}(x))[v_1, \ldots, v_{j-1}] 
		\colon v_1, \ldots, v_j \in \S^{n-1} \},
	\end{align*}
	while, by the Taylor expansion,
	$$
	(f^{(j)}(x+h) - f^{(j)}(x))[v_1, \ldots, v_j] = 
	f^{(j+1)}(x)[v_1, \ldots, v_j, h] + o(|h|_\infty)
	$$
	as $h \to 0$. Here, the $o$-term can be bounded by a quantity which is 
	independent of $v_1, \ldots, v_j \in \S^{n-1}$. As a consequence,
	\begin{align*}
		&
		\limsup_{h \to 0} 
		\frac{|\lvert f^{(j)}(x+h) \rvert_{\mathrm{op}} - \lvert f^{(j)}(x) \rvert_{\mathrm{op}}|}{|h|}\\
		\le \; &\sup \{ f^{(j+1)}(x)[v_1, \ldots, v_j, v_{j+1}] \colon v_1, \ldots, v_{j+1} \in \S^{n-1}\}
		= \lvert f^{(j+1)}(x) \rvert_{\mathrm{op}}.\qedhere
	\end{align*}
\end{proof}

The proof of Theorem \ref{thm:HigherOrderLSI} now reduces to a simple plug-in into Theorem \ref{thm:HigherOrder}. Finally, to see Proposition \ref{prop:HWILSU} we need to check \eqref{eq:ass3} together with some simple calculations.

\begin{proof}[Proof of Proposition \ref{prop:HWILSU}]
First note that in the notation of Corollary \ref{cor:HWI}, $\diamondsuit = \mathrm{op}$ and $\square = \mathrm{HS}$. Moreover,
    \[
    \partial_i f(x) = 2\sum_{j=1}^n a_{ij}x_j, \qquad \partial_{ij} f(x) = 2a_{ij}.
    \]
    In particular, $f''(x) = 2A$. In particular, each $\partial_i f$ is centered, so that by the Poincar\'{e} inequality (applied to each summand $\partial_i f$),
    \[
    \int \sum_{i=1} ^n |\partial_i f|^2 d\mu \le \sigma^2 \int \sum_{i,j=1}^n (\partial_{ij}f(x))^2 d\mu = \sigma^2 \int \lvert f''(x) \rvert_\mathrm{HS}^2 d\mu,
    \]
    which establishes \eqref{eq:ass3} and hence concludes the proof.
\end{proof}

\subsection{The Euclidean sphere}\label{sec:EuclSph}

One of the classical (and historically first) objects in the theory of concentration of measure is the Euclidean unit sphere
\[
\S^{n-1} = \{x \in \mathbb{R}^n \colon |x| = 1\},\qquad n \ge 2,
\]
equipped with the uniform measure $\nu_n$. The latter may be realized as the normalization of an $n$-dimensional standard Gaussian. Indeed, if $Z=(Z_1,\ldots,Z_n) \in \R^n$ is a random vector with i.i.d.\ standard Gaussian (i.\,e., $\mathcal{N}(0,1)$) components, then $Z/\lvert Z \rvert$ has distribution $\nu_n$, and $Z/\lvert Z \rvert$ and $\lvert Z \rvert$ are independent. From a different perspective, this means decomposing the $n$-dimensional standard Gaussian distribution into a radial and a directional component which are independent of each other.

We will formulate several higher order concentration results for functions on the sphere. In this context, a central observation is that $\nu_n$ satisfies a log-Sobolev inequality, so that we may lean on the results established in the previous section.

More precisely, as established in \cite{MW82}, $\nu_n$ satisfies a log-Sobolev inequality with constant $1/(n-1)$ with respect to the intrinsic gradient $\nabla_S$ (or, more generally, the generalized modulus of the gradient induced by the Riemannian metric on $\S^{n-1}$, cf.\ \eqref{eq:intgrad} and the discussion thereafter):
\begin{equation}\label{eq:LSISphere}
	\mathrm{Ent}_{\nu_n} (f^2) \le \frac{2}{n-1} \int_{\S^{n-1}} \lvert \nabla_S f \rvert_2^2 d\nu_n.
\end{equation}
To derive an explicit formula for $\nabla_S f$, recall that the tangent space in $\theta \in \S^{n-1}$ is given by $T_\theta\S^{n-1} = \theta^T = \{x \in \R^n \colon \langle x, \theta \rangle = 0\}$. Therefore, according to \eqref{eq:intgrad}, if $f$ is a smooth function in a neighbourhood of $\S^{n-1}$, its intrinsic gradient is given by
\begin{equation}\label{eq:FormelSphereGradient}
\nabla_{\S^{n-1}} f(\theta) = \nabla_S f(\theta) = \nabla f(\theta) - \langle \nabla f(\theta), \theta \rangle \theta.
\end{equation}
By orthogonality, it clearly holds that
\begin{equation}\label{eq:intrgradmin}
\lvert \nabla_S (\theta) \lvert \le \lvert \nabla f(\theta) \rvert
\end{equation}
for all $\theta \in \S^{n-1}$, so that the $\nabla_S$-LSI($1/(n-1)$) \eqref{eq:LSISphere} also implies a $\nabla$-LSI($1/(n-1)$).

In what follows, we will state three different types of higher order concentration results on $\S^{n-1}$ which involve different types of derivatives, reflecting the somewhat more involved nature of intrinsic calculus.

\subsubsection{Euclidean derivatives}
Once $f$ is defined in a neighbourhood of the sphere, we may simply take the usual Euclidean derivatives of $f$, which amounts to regarding $\nu_n$ as a probability measure on $\R^n$.

\begin{proposition}\label{prop:HigherOrderSphereEucl}
Whenever $f$ is a real-valued $\mathcal{C}^d$-function defined in a neighbourhood of the sphere, Theorem \ref{thm:HigherOrderLSI} and Proposition \ref{prop:HWILSU} continue to hold for $\mu = \nu_n$ with $\sigma^2 = 1/(n-1)$.
\end{proposition}

By Proposition \ref{prop:HigherOrderSphereEucl}, we retrieve \cite[Th.\ 1.13]{BGS19}. In applications, this approach can already be sufficient.

\subsubsection{Spherical derivatives}
Obviously, one would like to have a version of Proposition \ref{prop:HigherOrderSphereEucl} where the Euclidean derivatives are replaced by spherical (intrinsic) derivatives, which can be much sharper (cf. \eqref{eq:intrgradmin}) and only depend on the values of $f$ on $\S^{n-1}$.

To this end, we focus on the case $d=2$ and introduce a spherical Hessian. We first provide a purely intrinsic definition. For any $\mathcal{C}^2$-smooth function $f \colon \mathbb{R}^n \to \mathbb{R}$ at a given point $\theta \in \S^{n-1}$, consider the Taylor expansion up to the quadratic term
\begin{equation}\label{eq:Taylor2ndOrderSph}
	f(\theta') = f(\theta) + \langle\nabla_S f(\theta), \theta'-\theta\rangle + \frac{1}{2} \langle B(\theta'-\theta), \theta'-\theta \rangle + o\big(|\theta' - \theta|^2\big)
\end{equation}
as $\theta' \rightarrow \theta$, $\theta' \in \S^{n-1}$, where $B \in \mathbb{R}^{n \times n}$ is some matrix. The collection of all $B$'s satisfying \eqref{eq:Taylor2ndOrderSph} represents an affine subspace of $\mathbb{R}^{n \times n}$. Therefore, among all of them, there exists a unique matrix of smallest Hilbert--Schmidt norm. It is called the (intrinsic) second derivative of $f$ at the point $\theta$ and will be denoted $f_S''(\theta)$.

If $f$ is $\mathcal{C}^2$-smooth function on some open neighbourhood of $W_{n,k}$, by the usual (Euclidean) Taylor expansion, \eqref{eq:Taylor2ndOrderSph} holds with
\begin{equation}\label{eq:MatrixBSph}
	B = f''(\theta) - \langle \theta, \nabla f(\theta)\rangle I_n,
\end{equation}
where $f''(\theta)$ is the Euclidean Hessian of $f$ at $\theta$. In general, given $C \in \mathbb{R}^{n \times n}$, the matrix $B - C$ satisfies \eqref{eq:Taylor2ndOrderSph} if and only if
\[
\langle C(\theta' - \theta), \theta'-\theta\rangle = o\big(|\theta'-\theta|^2\big)
\]
for $\theta' \to \theta$, $\theta' \in \S^{n-1}$. Similarly to the first order case, this is equivalent to $\langle Cx, x \rangle = 0$ for all $x \in T_\theta$. This condition defines a linear subspace $L$ of $\mathbb{R}^{n \times n}$, and the minimization problem translates into
\[
\lvert B - C \rvert \to \mathrm{min}\qquad \text{over all $C \in L$,}
\]
which is solved uniquely for the orthogonal projection of $B$ onto the linear space $L^\perp$ of all matrices orthogonal to $L$. Moreover, we may restrict ourselves to symmetric matrices (by symmetry of the Euclidean Hessian). As a result, we obtain the following description, which goes back to \cite[Prop.\ 8.1]{BCG17}.

\begin{proposition}\label{2ndOrDerSph}
	The intrinsic second derivative of $f$ at each $\theta \in \S^{n-1}$ is the symmetric matrix which is given by the orthogonal projection
	\[
	f_S''(\theta) = P_{\theta^\perp} B, \qquad B = f''(\theta) - \langle \theta, \nabla f(\theta)\rangle I_n,
	\]
	to the orthogonal complement of the linear subspace $L = L_\theta$ of all symmetric matrices $C$ in $\mathbb{R}^{n \times n}$ such that $\langle Cx, x \rangle = 0$ for all $x \in T_\theta$. Equivalently,
	\[
	f_S''(\theta) = P_{\theta^\perp} B P_{\theta^\perp}.
	\]
\end{proposition}

In particular, we have the contraction property
\[
\vert f_S''(\theta)\rvert_\mathrm{HS} \, \leq \, \lvert f''(\theta) - \langle \theta, \nabla f(\theta)\rangle I_n\rvert_{\mathrm{HS}},
\]
which also holds for the operator norm.

We can now formulate a second order concentration result on the sphere.

\begin{proposition}\label{prop:HigherOrderSphereIntr}
	Let $f \colon \S^{n-1} \to \mathbb{R}$ be a $\mathcal{C}^2$-smooth function with $\E f =0$.
	\begin{enumerate}
		\item Assuming $\lVert \nabla_S f \rVert_{\mathrm{HS},2} \le 1/\sqrt{n-1}$ and $\lVert f''_S \rVert_{\mathrm{op},\infty} \le 1$, we have
		\[
		\int_{\S^{n-1}} \exp\Big(\frac{n-1}{32e} |f|\Big) d\nu_n \le 2.
		\]
		\item For $C = 16e^2/\log(2)$ and any $t \ge 0$,
		\[
		\nu_n(|f-\nu_n(f)| \ge t) \le 2 \exp\Big(-\frac{n-1}{C}\min\Big(\frac{t^2}{\lVert \nabla_S f \rVert_{\mathrm{HS},2}^2}, \frac{t}{\lVert f_S'' \rVert_{\mathrm{op},\infty}}\Big)\Big).
		\]
		\item If $\E \nabla_S f = 0$, the bounds in (1) and (2) continue to hold with $\lVert \nabla_S f \rVert_{\mathrm{HS},2}^2$ replaced by $\lVert f_S'' \rVert_{\mathrm{HS},2}^2/(n+2)$.
	\end{enumerate}
\end{proposition}

In particular, as it has been discussed in \cite{BCG17}, the property $\E \nabla_S f = 0$ is equivalent to $f$ being orthogonal to all linear functions, so that we get back \cite[Th.\ 1.1]{BCG17}.

To prove Proposition \ref{prop:HigherOrderSphereIntr} (1) \& (2), one needs to check Assumptions \ref{ass:Set1} almost in the same way as in the proof of Theorem \ref{thm:HigherOrderLSI} just with the obvious modifications due to the use of intrinsic derivatives.

To see Assumption \ref{ass:Set1} (1), note that by \eqref{eq:AS} (whose proof remains valid since it essentially only uses the chain rule \eqref{eq:chainrule}) and \eqref{eq:1fLSI}, we obtain
\begin{equation*} 
\lVert g - \E g \rVert_r \le \frac{1}{\sqrt{n-1}} r^{1/2} \lVert \nabla_S g \rVert_r
\end{equation*}
for any $r \ge 2$. Due to the use of the intrinsic gradient, Assumption \ref{ass:Set1} (2) is more subtle. Indeed, by somewhat more extensive calculations, it has been shown in \cite[Le.\ 3.1]{BCG17} that
\begin{equation}\label{eq:BCGitgrad}
	\lvert\nabla_S\lvert\nabla_S f \rvert \rvert \le \lvert f''_S \rvert_\mathrm{op}.
\end{equation}

These two properties hence establish Proposition \ref{prop:HigherOrderSphereIntr} (1) \& (2). Moreover, (3) is a consequence of Corollary \ref{cor:HWI}, for which we need to verify \eqref{eq:ass3}. Setting $\diamondsuit = \mathrm{op}$ and $\square = \mathrm{HS}$, the relevant calculations are due to \cite[Prop.\ 5.1]{BCG17}.

Summing up, we have thus (re-)established second order concentration results on the sphere in the spirit of \cite{BCG17}. The drawback of this approach is that it is considerably difficult to extend it to higher orders $d \ge 3$. Indeed, it is not quite immediate even how to define an intrinsic third order tensor of derivatives for the sphere.

\subsubsection{Spherical partial derivatives}
It is possible to formulate a sort of ''intermediate`` result between Proposition \ref{prop:HigherOrderSphereEucl} and Proposition \ref{prop:HigherOrderSphereIntr} which holds for every fixed order $d$ and works with a notion of derivatives which is stronger than the standard Euclidean but weaker than spherical derivatives.

Here, the key notion are spherical partial derivatives, which are introduced as follows. Given a differentiable function $f \colon \S^{n-1} \to \mathbb{R}$, the spherical partial derivatives are given by $D_i f = \langle \nabla_S f, e_i \rangle$, $i = 1, \ldots, n$, where $e_i$ denotes the $i$-th standard unit vector in $\mathbb{R}^n$. Higher order spherical partial derivatives are defined by iteration, e.\,g.,  $D_{ij} f = \langle \nabla_S \langle \nabla_S f, e_j \rangle, e_i \rangle$ for any $1 \le i,j \le n$. Note that in general, $D_{ij}f \ne D_{ji} f$. If $f \in \mathcal{C}^d(\S^{n-1})$, we denote by $D^{(d)} f (\theta)$ the hyper-matrix of the spherical partial derivatives of order $d$, i.\,e.,
\begin{equation*}
	(D^{(d)} f(\theta))_{i_1 \ldots i_d} = D_{i_1 \ldots i_d} f(\theta), \qquad
	\theta \in S^{n-1}.
\end{equation*}

\begin{proposition}\label{prop:HigherOrderSpherePart}
For any $f \in \mathcal{C}^d(\S^{n-1})$, Theorem \ref{thm:HigherOrderLSI} and Proposition \ref{prop:HWILSU} hold with $\sigma^2 = 1/(n-1)$ and with the Euclidean derivatives $f^{(j)}$ replaced by spherical partial derivatives $D^{(j)} f$, $j=1, \ldots, d$.
\end{proposition}

Proposition \ref{prop:HigherOrderSpherePart} gives back \cite[Th.\ 1.15]{BGS19} (including a multilevel tail version).

Clearly, the proof of Proposition \ref{prop:HigherOrderSpherePart} reduces to establishing Assumption \ref{ass:Set1} (2) for spherical partial derivatives. By choosing suitable homogeneous extensions of $f$ and its derivatives, it has been shown in \cite[Le.\ 5.1]{BGS19} that given a $\mathcal{C}^d$-smooth function $f \colon \S^{n-1} \to \mathbb{R}$, it holds for for all $\theta \in \S^{n-1}$ and all $j=1, \ldots, d-1$,
	\begin{equation}
	\label{eq:itGradHesssph}
	|D\lvert D^{(j)}f(\theta) \rvert_\mathrm{op}| \le \lvert D^{(j+1)}f(\theta) \rvert_\mathrm{op}.
	\end{equation}
In particular, \ref{eq:itGradHesssph} establishes Assumption \ref{ass:Set1} (2) for spherical partial derivatives, which proves Proposition \ref{prop:HigherOrderSpherePart}.

\subsection{Poincar\'{e}-type inequalities}

Again considering $\mathcal{X} = \R^n$, let us discuss the case where $\mu$ satisfies a Poincar\'e-type inequality \eqref{eq:genPI} with respect to the Euclidean gradient, i.\,e., $\nabla$-PI$(\sigma^2)$. As we have already noted, Poincar\'e-type inequalities are a more general but also weaker concept than log-Sobolev inequalities. For instance, Poincar\'{e}-type inequalities only imply sub-exponential tails for Lipschitz functions.

It is not hard to verify Assumptions \ref{ass:Set1} if $\mu$ satisfies a Poincar\'{e}-type inequality. To see (1), we recall that by \cite[Prop.\ 6.2.1]{BCG23}, for every locally Lipschitz function $g$, it holds that
\begin{equation}
\lVert g - \E g \rVert_r \le \frac{\sigma}{\sqrt{2}} r \lVert \nabla g \rVert_r
\end{equation}
for any $r \ge 2$, which establishes (1) with $L=1/\sqrt{2}$ and $p=1$. Moreover, (2) can be dealt with exactly in the same way as in the case of log-Sobolev inequalities, including Lemma \ref{lem:itGradHessLSI}.

In particular, we may recover Th.\ 6.1.1 and Cor.\ 6.1.3 from \cite{GS19}, which read as follows in our framework.

\begin{satz}\label{thm:HigherOrderPI}
	Let $\mu$ be a probability measure on $\R^n$ which satisfies a Poincar\'{e}-type inequality with constant $\sigma^2 > 0$, and let $f\colon \R^n \to \R$ be a $\mathcal{C}^d$-smooth function with $\E f = 0$.
	\begin{enumerate}
		\item Assuming that $\lVert f^{(j)} \rVert_{\mathrm{op},1} \le \sigma^{d-j}$ for all $j=1,\ldots,d-1$ and $\lVert f^{(d)} \rVert_{\mathrm{op},\infty} \le 1$, it holds that
		\[
		\int \exp\Big\{\frac{c}{\sigma} |f|^{1/d} \Big\} d\mu \le 2,
		\]
		where we may choose $c=2^{1/(2d)}/(4e)$.
		\item For any $t\ge 0$, it holds that
		\[
		\mu(|f| \ge t)\le
		2 \exp\Big\{- \frac{C_d}{d\sigma} \min\Big(\min_{j=1, \ldots,d-1} \Big(\frac{t}{\lVert f^{(j)} \rVert_{\mathrm{op},1}}\Big)^{1/j}, \Big(\frac{t}{\lVert f^{(d)} \rVert_{\mathrm{op},\infty}}\Big)^{1/d}\Big)\Big\},
		\]
		where we may choose $C = \log(2)/(\sqrt{2}e)$.
	\end{enumerate}
\end{satz}

We may furthermore derive the following Hanson--Wright type inequality. For a proof, see the proof of Proposition \ref{prop:HWILSU}.

\begin{proposition}\label{prop:HWIPU}
    Let $\mu$ be a probability measure on $\R^n$ which satisfies a Poincar\'{e}-type inequality as well as $\int x_i d\mu = 0$ for any $i=1, \ldots,n$, and define $f\colon \R^n \to \R$ by $f(x) := x^TAx$, where $A=(a_{ij})_{ij}$ is a symmetric $n \times n$ matrix. Then,
    \[
    \mu(|f - \E f| \ge t) \le 2\exp\Big\{-c \min \Big(\frac{t}{\sigma^2\lvert A \rvert_{\mathrm{HS}}}, \Big(\frac{t}{\sigma^2\lvert A \rvert_{\mathrm{op}}}\Big)^{1/2}\Big)\Big\}
    \]
    for any $t \ge 0$.
\end{proposition}

Note that in presence of a Poincar\'{e}-type inequality, the coordinates $x_i$ are sub-exponential random variables. Therefore, if they are also independent, it is known due to \cite{KL15} and \cite{GSS21b} that in this situation, the result from Proposition \ref{prop:HWIPU} can still be sharpened by invoking different techniques. On the other hand, Proposition \ref{prop:HWIPU} again allows for measures with dependent coordinates.

\subsection{\texorpdfstring{$\mathrm{LS}_q$}{LS\_q}-inequalities}

In terms of Lipschitz functions, log-Sobolev inequalities give rise to the ``classical'' situation of sub-Gaussian tails, while Poincar\'{e}-type inequalities result in sub-exponential (and thus heavier) tails. There is another concept of functional inequalities which generalize log-Sobolev inequalities, yielding sub-Gaussian or lighter tails.

Let $\mu$ be a probability measure on $(\mathcal{X},\mathcal{B})$, and let $q \in [1,2]$. Then, we say that $\mu$ satisfies an $\mathrm{LS}_q$-inequality (with respect to $\Gamma$) with constant $\sigma^q > 0$ (in short: $\mu$ satisfies $\Gamma$-$\mathrm{LS}_q(\sigma^q)$) if
\begin{equation}
	\label{eq:LSq}
	\mathrm{Ent}_{\mu}(|f|^q) \le \sigma^q \int \Gamma(f)^q d\mu
\end{equation}
for all $f \in \mathcal{F}$. For $q=2$, we get back the usual notion of logarithmic Sobolev inequalities (up to normalization constant), of which $\mathrm{LS}_q$-inequalities may be regarded as an extension. For a broader view on $\mathrm{LS}_q$-inequalities, cf.\ e.\,g.\ the monograph \cite{BZ05}.

By \cite[Th.\ 2.1]{BZ05}, \eqref{eq:LSq} implies the Poincar\'{e}-type (or $\mathrm{SG}_q$-) inequality
\begin{equation}\label{eq:SGq}
\int \Big| f - \int f d\mu \Big|^q d\mu \le \frac{4\sigma^q}{\log(2)} \int \Gamma(f)^q d\mu.
\end{equation}
As for further implications, it was shown in \cite[Prop.\ 2.3 \& Cor.\ 2.5]{BZ05} that $\mathrm{SG}_q (\sigma^q)$ implies $\mathrm{SG}_2 (\sigma^2)$ (up to constant). Generalizing these arguments, one may prove that for $1 \le q \le q' \le 2$, $\mathrm{SG}_q(\sigma^q)$ implies $\mathrm{SG}_{q'}(\sigma^{q'})$ (again, up to constant) and even that $\mathrm{LS}_q(\sigma^q)$ implies $\mathrm{LS}_{q'}(\sigma^{q'})$ with respect to the same difference operator~$\Gamma$.

Typically, $\mathrm{LS}_q$-inequalities are considered on metric probability spaces $(\mathcal{X}, d,\mu)$. In the sequel, we will again focus on the situation where $\mathcal{X} = \R^n$. Here, it is natural to consider $\mathrm{LS}_q$-inequalities for $\R^n$ being equipped with the $\ell_p^n$-norm $\lvert x \rvert_p := (|x_1|^p + \ldots + |x_n|^p)^{1/p}$ ($x \in \R^n$, $p \ge 2$), hence $\Gamma(f) = |\nabla f|_q$, (more generally, $\Gamma(f)$ is the generalized modulus of the gradient \eqref{eq:genmod}). In this context, we will always regard $q \in [1,2]$ as the H\"older conjugate of $p \in [2, \infty]$ and vice versa, i.\,e., $q=p/(p-1)$ and $p=q/(q-1)$.

Classical examples of measures on $\R^n$ satisfying $\mathrm{LS}_q$-inequalities are the so-called $p$-generalized Gaussian distributions $\mathcal{N}_p$ for $p \ge 2$. For $n=1$, these are the probability measures with Lebesgue density
\[
f_p(x) := c_p e^{-|x|^p/p},\qquad \text{where}\qquad c_p := \frac{1}{2p^{1/p}\Gamma(1 + \frac{1}{p})},
\]
and for $n \ge 2$, $\mathcal{N}_p$ is the $n$-fold product of this distribution. For instance, one may show that for any real $p \ge 2$, the $p$-generalized Gaussian distribution $\mathcal{N}_p$ satisfies an $\mathrm{LS}_q$-inequality with a constant $\sigma^q = 2^qq^{q-1}$, cf.\ e.\,g.\ \cite[Prop.\ 3.1]{GS24}.

By \cite[Th.\ 6.1]{Bob10}, the $\mathrm{LS}_q$-inequality implies the moment (or $L^r$-norm) inequality
\begin{equation*}
	\lVert g \rVert_r^q \le \lVert g \rVert_q^q + \sigma^q \frac{(\frac{r}{q})^{q-1}-1}{q-1} \lVert \nabla g \rVert_{q,r}^q,
\end{equation*}
valid for any $r \ge q$ and any locally Lipschitz function $g \colon \R^n \to \R$ (recall the notation $\lVert \nabla g \rVert_{q,r} := \lVert \lvert \nabla g \rvert_q \rVert_r$). Assuming $\E g = 0$, we may further apply the $\mathrm{SG}_q$-inequality \eqref{eq:SGq}. Noting that
\[
\frac{(\frac{r}{q})^{q-1}-1}{q-1} + \frac{4}{\log(2)} \le \frac{4}{\log(2)q^{q-1}(q-1)}r^{q-1} \le \frac{4}{\log(2)(q-1)^q}r^{q-1},
\]
this leads to
\begin{equation}\label{eq:LSqLp}
	\lVert g - \E g \rVert_r \le \frac{4^{1/q}(p-1)}{\log^{1/q}(2)} \sigma r^{1/p} \lVert \nabla g \rVert_{q,r}
\end{equation}
for any $r \ge q$. This establishes Assumption \ref{ass:Set1} (1) with $r_0=q$ and $L = 4^{1/q}(p-1)/\log^{1/q}(2)$.

Moreover, Assumption \ref{ass:Set1} (2) follows very similarly as in the case of (classical) log-Sobolev inequalities. Again, we consider tensors of $j$-fold (partial) derivatives $f^{(j)}(x)$. By considering $f^{(j)}(x)$ as a symmetric multilinear $j$-form, we moreover define operator-type norms with respect to the $\ell_p^n$-norms by
\begin{equation}
	\label{eq:Operatornormq}
	|f^{(j)}(x)|_{\mathrm{op}(q)} = \sup \left\{ f^{(j)}(x)[v_1, \ldots, v_j] \colon v_1, \ldots, v_j \in \S_p^{n-1}\right\}.
\end{equation}
In particular, $|f^{(1)}|_{\mathrm{op}(q)} = |\nabla f|_q$ is the $\ell_q^n$-norm of the gradient.

\begin{lemma}[\cite{GS24}, Le.\ 3.3]
	\label{lem:itGradHessLSq}
	Given a $\mathcal{C}^j$-smooth function $f \colon \R^n \to \mathbb{R}$, $j \in \mathbb{N}$, for any $x \in \R^n$ it holds that
	\[
	|\nabla \lvert f^{(j-1)}(x) \rvert_{\mathrm{op}(q)}|_q \le 
	\lvert f^{(j)}(x) \rvert_{\mathrm{op}(q)}.
	\]
\end{lemma}

\begin{proof}
	This amounts to following the proof of Lemma \ref{lem:itGradHessLSI}, replacing the Euclidean operator norms by $q$-operator norms and hence vectors $v_i \in \S^{n-1}$ by vectors $v_i \in \S_p^{n-1}$.
\end{proof}

Altogether, we thus get back \cite[Th.\ 3.2]{GS24}. See also the discussion therein for a comparison to the results based modified log-Sobolev inequalities from \cite{ALM18}.

\begin{satz}\label{thm:HigherOrderLSq}
	Let $\mu$ be a probability measure on $\R^n$ which satisfies an $\mathrm{LS}_q$-inequality with constant $\sigma^q > 0$, and let $f\colon \R^n \to \R$ be a $\mathcal{C}^d$-smooth function with $\E f = 0$.
	\begin{enumerate}
		\item Assuming that $\lVert f^{(j)} \rVert_{\mathrm{op}(q),1} \le \sigma^{d-j}$ for all $j=1,\ldots,d-1$ and $\lVert f^{(d)} \rVert_{\mathrm{op}(q),\infty} \le 1$, it holds that
		\[
		\int \exp\Big\{\frac{c_{p,d}}{\sigma^p} |f|^{p/d} \Big\} d\mu \le 2,
		\]
		where a possible choice of the constant $c_{p,d}$ is given by
		\[
		c_{p,d} = \frac{2^{2p-3}(q^{1/p}-1)^p(p-1)^p}{eq \log^{p-1}(2)\max(q, p/d)}.
		\]
		\item For any $t\ge 0$, it holds that
		\[
		\mu(|f| \ge t)\le
		2 \exp\Big\{- \frac{C_p}{(d\sigma)^p} \min\Big(\min_{j=1, \ldots,d-1} \Big(\frac{t}{\lVert f^{(j)} \rVert_{\mathrm{op}(q),1}}\Big)^{p/j}, \Big(\frac{t}{\lVert f^{(d)} \rVert_{\mathrm{op}(q),\infty}}\Big)^{p/d}\Big)\Big\},
		\]
		where a possible choice of the constant $C_p$ is given by
		\[
		C_p = \frac{4^{p-1}(p-1)^p}{q\log^{p-2}(2) e^p}.
		\]
	\end{enumerate}
\end{satz}

We moreover have a Hanson--Wright type inequality in this situation.

\begin{proposition}[\cite{GS24}, Prop.\ 3.4]\label{prop:HWILSq}
    Let $\mu$ be a probability measure on $\R^n$ which satisfies an $\mathrm{LS}_q$-inequality as well as $\int x_i d\mu = 0$ for any $i=1, \ldots,n$, and define $f\colon \R^n \to \R$ by $f(x) := x^TAx$, where $A=(a_{ij})_{ij}$ is a symmetric $n \times n$ matrix. Then,
    \[
    \mu(|f - \E f| \ge t) \le 2\exp\Big\{-c_p\min \Big(\Big(\frac{t}{\sigma^2\lvert A \rvert_{\mathrm{HS}(q)}}\Big)^p, \Big(\frac{t}{\sigma^2\lvert A \rvert_{\mathrm{op}(q)}}\Big)^{p/2}\Big)\Big\}
    \]
    for any $t \ge 0$.
\end{proposition}

To show Proposition \ref{prop:HWILSq}, one follows the lines of the proof of Proposition \ref{prop:HWILSU}, adapting the norms accordingly and in particular noting that by the $\mathrm{SG}_q$-inequality \eqref{eq:SGq} (applied to each summand $\partial_i f$),
    \[
    \int \sum_{i=1} ^n |\partial_i f|^q d\mu \le \frac{4\sigma^q}{\log(2)} \int \sum_{i,j=1}^n |\partial_{ij} f(x)|^q d\mu.
    \]

\subsection{\texorpdfstring{$\ell_p^n$}{l\_p\^n}-spheres}

In view of Sections \ref{sec:LSIs} and \ref{sec:EuclSph}, there is a close connection of (higher order) concentration results on the Euclidean space in presence of a log-Sobolev inequality and on the Euclidean sphere $\S^{n-1}$. Similar observations also hold for $\ell_p^n$-spheres, where (higher order) concentration of measure results can be obtained by means of $\LSq$-inequalities.

For $p \in [1,\infty)$ and $n \in \N$, let $\S_p^{n-1} := \{ x \in \R^n \colon \lvert x \rvert_p = 1\}$ the unit sphere with respect to $\ell_p^n$. On $\S_p^{n-1}$, there are two classical probability measures. The first one is the surface measure $\nu_{p,n}$, defined as the normalized $(n-1)$-dimensional Hausdorff measure on $\S_p^{n-1}$. The second one is the cone probability measure $\mu_{p,n}$, defined by
\[
\mu_{p,n}(A) := \frac{\vol_n(\{t\theta \colon t \in [0,1], \theta \in A\})}{\vol_n(\mathbb{B}_p^n)}
\]
for any Borel set $A \subset \S_p^{n-1}$, where $\mathbb{B}_p^n$ is the $\ell_p^n$-unit ball in $\R^n$. By definition, $\mu_{p,n}(A)$ is the normalized volume of the intersection of $\mathbb{B}_p^n$ with the cone which intersects $\S_p^{n-1}$ in $A$. For $p=1$ and $p=2$, $\nu_{p,n}$ and $\mu_{p,n}$ agree, and we get back the uniform measure $\nu_n$ on $\S_2^{n-1} = \S^{n-1}$. See e.\,g.\ \cite[Sect.\ 9.3.1]{PTT19} for details.

It is possible to establish $\mathrm{LS}_q$-inequalities for the cone measure on $\S_p^{n-1}$ for $p \ge 2$ by using that if $Z = (Z_1,\ldots,Z_n)$ is a random vector with i.i.d.\ $p$-generalized Gaussian components $Z_i \sim \mathcal{N}_p$, it holds that $Z/\lvert Z \rvert_p \sim \mu_{p,n}$, and $Z/\lvert Z \rvert_p$ is independent of $\lvert Z \rvert_p$. This fact goes back to \cite{SZ90,RR91}, cf.\ also \cite[Th.\ 9.3.2]{PTT19}. The $\mathrm{LS}_q$-inequality can then be ``inherited'' from the one for $p$-generalized Gaussians.

Here, one may introduce an intrinsic gradient by noting that $\S_p^{n-1}$ is a (Riemannian) submanifold of $\R^n$ given by the zero set of the function $F_p(x) := |x_1|^p + \ldots + |x_n|^p - 1$. It holds that $\nabla F_p(x) = p(\sign(x_1)|x_1|^{p-1}, \ldots, \sign(x_n)|x_n|^{p-1})^T =: px^{p-1}$, where $\mathrm{sign}(x)$ denotes the sign function, and thus, the tangent space in $\theta = \theta_p \in \S_p^{n-1}$ is given by
\[
T_\theta\S_p^{n-1} = \mathrm{ker} F_p'(\theta) = \{ x \in \R^n \colon \langle x, \theta^{p-1} \rangle = 0\}.
\]
In particular, let $f$ be a smooth function defined on a neighbourhood of $\S_p^{n-1}$, let $\theta \in \S_p^{n-1}$, and denote by $P_\theta$ the orthogonal projection onto $T_\theta\S_p^{n-1}$. Then, by \eqref{eq:intgrad}, the intrinsic gradient of $f$ at $\theta$ is given by
\[
\nabla_{\S_p^{n-1}} f(\theta) = \nabla_S f(\theta) = P_\theta \nabla f(\theta).
\]

\begin{satz}[\cite{GS24}, Th.\ 4.2]\label{thm:LSICone}
	For any real $p \ge 2$, $\mu_{p,n}$ satisfies an $\mathrm{LS}_q$-inequality with constant of order $n^{-1/(p-1)}$. More precisely, for any smooth enough $f \colon \S_p^{n-1} \to \R$, it holds that
	\begin{align*}
		\mathrm{Ent}_{\mu_{p,n}}(|f|^q) &\le 4^q q^{q-1} \frac {\Gamma(\frac{n-q}{p})}{p^{q/p}\Gamma(\frac n p)} \int_{\S_p^{n-1}} \lvert \nabla_{S} f(\theta) \rvert_q^q d \mu_{p,n}\\
		&\le 3 \cdot 4^q q^{q-1} n^{-\frac{1}{p-1}} \int_{\S_p^{n-1}} \lvert \nabla_S f(\theta) \rvert_q^q d \mu_{p,n},
	\end{align*}
	where the second inequality is valid for any $n \ge 3$. The same inequalities remain true if we assume $f$ to be defined and smooth in a neighbourhood of $\S_p^{n-1}$ and replace $\nabla_S f$ by the usual (non-intrinsic) gradient $\nabla f$.
\end{satz}

For $p=2$, Theorem \ref{thm:LSICone} recovers the usual logarithmic Sobolev inequality on the sphere with Sobolev constant of order $n^{-1}$.

As seen in the previous section, Theorem \ref{thm:LSICone} essentially establishes Assumption \ref{ass:Set1} (1) for the cone measure by means of plugging into \eqref{eq:LSqLp}. Assumption (2) is more delicate. In principle, the intrinsic gradient $\nabla_S f$ is the natural gradient of a smooth function $f$ on $\S_p^{n-1}$ and independent of any sort of smooth extension of $f$ onto $\R^n$. Nevertheless, working with the usual gradient $\nabla f$ can turn out to be convenient after all. This has two main reasons: first, even for second order derivatives, calculating an intrinsic ``Hessian'' becomes remarkably involved if $p \ne 2$. Moreover, in general it is not even clear whether $|\nabla_S f|_q \le |\nabla f|_q$ (unless $p =2$).

Against this background, the easiest and most convenient way to obtain higher order concentration results for the cone measure $\mu_{p,n}$ on $\S_p^{n-1}$ is to regard $\mu_{p,n}$ as a measure on $\R^n$ which satisfies an $\mathrm{LS}_q$-inequality with respect to the usual gradient $\nabla f$ on $\R^n$. Combining Theorem \ref{thm:HigherOrderLSq} with Theorem \ref{thm:LSICone} (2) (hence, setting $\sigma^q := 3 \cdot 4^q q^{q-1} n^{-1/(p-1)}$), we obtain the following result.

\begin{satz}\label{thm:HOSph}
	Let $n \ge 3$, and let $f$ be a function which is $\mathcal{C}^d$ in some neighbourhood of $\S_p^{n-1}$ and with $\mu_{p,n}$-mean zero.
	\begin{enumerate}
		\item Assuming that $\lVert f^{(k)} \rVert_{\mathrm{op}(q),q} \le (3^{1/q} \cdot 4 q^{1/p} n^{-1/p})^{d-k}$ for all $k=1,\ldots,d-1$ and $\lVert f^{(d)} \rVert_{\mathrm{op}(q),\infty} \le 1$, it holds that
		\[
		\int \exp\Big\{c_pn |f|^{p/d} \Big\} d\mu_{p,n} \le 2,
		\]
		where a possible choice of the constant $c_{p,d}$ is given by
		\[
		c_p = \frac{(q^p-1)^p(p-1)^p}{2 \cdot 3^{p-1}\log^{p-1}(2)q^{p^2+2}e}.
		\]
		\item For any $t\ge 0$, it holds that
		\begin{align*}
			&\mu_{p,n}(|f| \ge t)\le\\
			&2 \exp\Big\{- \frac{C_pn}{d^p} \min\Big(\min_{k=1, \ldots,d-1} \Big(\frac{t}{\lVert f^{(k)} \rVert_{\mathrm{op}(q),q}}\Big)^{p/k}, \Big(\frac{t}{\lVert f^{(d)} \rVert_{\mathrm{op}(q),\infty}}\Big)^{p/d}\Big)\Big\},
		\end{align*}
		where a possible choice of the constant $C_{p,d}$ is given by
		\[
		C_p = \frac{(p-1)^p}{4 \cdot 3^{p-1}q^2\log^{p-2}(2) e^p}.
		\]
	\end{enumerate}
\end{satz}

For futher background, cf.\ the presentation in \cite{GS24}, where it is also discussed how to extend these results to the surface measure (at least partially).

\subsection{Stiefel manifolds}

There are various other possible generalizations of the Euclidean sphere $\S^{n-1}$, for instance by the notion of Stiefel manifolds. For any natural numbers $k \le n$, the Stiefel manifold $W_{n,k}$ is the set of all $k$-tupels of orthonormal vectors in $\mathbb{R}^n$. In the sequel, we shall use the representation
\[
W_{n,k} = \{A \in \mathbb{R}^{n \times k} \colon A^TA = I_k \},
\]
where $I_k \in \mathbb{R}^{k \times k}$ denotes the identity matrix. $W_{n,k}$ is manifold of dimension $\mathrm{dim}(W_{n,k}) = nk - k(k+1)/2$, we have $W_{n,1} = \S^{n-1}$, $W_{n,n} = O(n)$ (the orthogonal group), and $W_{n,n-1}$ can be identified with the special orthogonal group $SO(n) = \{O \in O(n) \colon \mathrm{det}(O) = 1\}$.

We may equip $W_{n,k}$ with the subspace topology and distances inherited from $\mathbb{R}^{n \times k}$, including the scalar product $\langle A, B\rangle:=\mathrm{tr}(A^T B)$ on $\mathbb{R}^{n\times k}$ and the induced (Hilbert--Schmidt) norm $\lvert A \rvert^2 \equiv \lVert A\rVert_\mathrm{HS}^2$. In particular, $W_{n,k}$ is a compact topological space. The product of the orthogonal groups $O(n) \times O(k) $ acts transitively on $W_{n,k}$ by the two-sided multiplication $O_n \times O_k \mapsto O_nAO_k$, turning $W_{n,k}$ into a homogeneous space. Hence, it may be equipped with a unique invariant (Haar) probability measure $\mu_{n,k}$ in the sense that if $A \sim \mu_{n,k}$ (i.\,e., $A$ has distribution $\mu_{n,k}$), $O_nAO_m \sim \mu_{n,k}$ for any $O_n \in O(n)$, $O_m \in O(m)$. We call $\mu_{n,k}$ the uniform distribution on $W_{n,k}$. If $G = (G_{ij})_{i,j=1}^{m,n}$ is an $n \times m$ random matrix whose entries are i.i.d.\ standard normal, then $G(G^TG)^{-1/2} \sim \mu_{n,k}$, see \cite[Le.\ 3.1]{KPT20}.

Let us verify Assumptions \ref{ass:Set1} for the Haar measure on Stiefel manifolds. To this end, we first recall a notion of differentiability for matrix calculus. In general, if $A = (A_{ij})$ is an $n \times m$ matrix (i.\,e., $A \in \mathbb{R}^{n \times m}$), we may vectorize it by setting
\begin{equation}\label{eq:vec}
\mathrm{vec}(A)  := (A_{11}, \ldots, A_{n1}, A_{12}, \ldots, A_{n2}, \ldots, A_{1m},\ldots, A_{nm})^T,
\end{equation}
i.\,e.\ $\mathrm{vec}(A)$ is the vector in $\mathbb{R}^{nm}$ with $m$ $n$-blocks corresponding to the columns of $A$. Occasionally, we will also need the inverse operation of $\mathrm{vec}$, which we denote by $\mathrm{mat}$. If $f \colon \mathbb{R}^{n\times m} \to \mathbb{R}^{p\times q}$ is any differentiable (possibly matrix-valued) function of $X \in \mathbb{R}^{n \times m}$, we define
\begin{equation}\label{eq:matrixderiv}
Df(X) := D\mathrm{vec}(f(X)) := \frac{d f(X)}{dX} := \frac{d \mathrm{vec}(f(X))}{d \mathrm{vec}(X)},
\end{equation}
which is a $pq \times mn$ matrix.

To formally introduce intrinsic derivatives on Stiefel manifolds, recall that the tangent space in $A \in W_{n,k}$ is given by
\[
T_A:= \{  N \in \mathbb{R}^{n \times k} : A^T N + N^T A = 0\} = \{  N \in \mathbb{R}^{n \times k} : A \cs N = 0\},
\]
where for any $M, N \in \mathbb{R}^{n \times k}$, $M \cs N$ denotes the ``symmetric product''
\[
M \cs N := \frac{1}{2} (M^T N + N^T M),
\]
which is a symmetric $k \times k$ matrix (for $k=1$, this reduces to the Euclidean scalar product). Hence, $T_A$ is the set of all the matrices $N \in \mathbb{R}^{n \times k}$ such that $A^TN$ is antisymmetric.

Following \eqref{eq:intgrad}, let us now take any real-valued function $f$ which is defined and smooth in a neighbourhood of $W_{n,k}$. Denote by $Df(A) \in \mathbb{R}^{1 \times nk}$ its Euclidean derivative in $A \in W_{n,k}$ and set $\nabla f(A) := \mathrm{mat}(Df(A)^T) \in \mathbb{R}^{n \times k}$ (the usual gradient rewritten as a matrix). Then, $\nabla_W f(A)$ is given by the projection onto the tangent space $P_A \nabla f(A)$. Noting that the latter reads
\begin{equation}\label{eq:projfSt}
	\pi_AM := M - A(A \cs M),
\end{equation}
we obtain
\begin{equation}\label{eq:FormelStiefelGradient}
	\nabla_W f(A) = \nabla f(A) - A(A\cs \nabla f(A)).
\end{equation}

In particular, by the contractivity of orthogonal projections, $|\nabla_W f(A)| \leq |\nabla f(A)|$. The length of $\nabla_W f(A)$ agrees with \eqref{eq:genmod} for the Euclidean (Hilbert--Schmidt) metric $d(X,X') = |X-X'|$. Note we could also use the (point-dependent) canonical instead of the Euclidean metric, which leads to different notions of differentiability, cf.\ e.\,g.\ \cite{EAS98}.

Once again, the key observation in order to establish Assumption \ref{ass:Set1} (1) is that $\mu_{m,n}$ satisfies a log-Sobolev inequality.

\begin{proposition}[\cite{GS23}, Prop.\ 6.2]\label{prop:LSUSt}
	For any $k < n$, $\mu_{n,k}$ satisfies a logarithmic Sobolev inequality with constant $4/(n-2)$, i.\,e.\ for any $f \colon W_{n,k} \to \mathbb{R}$ sufficiently smooth,
		\[
		\mathrm{Ent}_{\mu_{n,k}} (f^2) \le \frac{8}{n-2} \int_{W_{n,k}} |\nabla_W f|^2 d\mu_{n,k}.
		\]
\end{proposition}

If $k=n$, $W_{n,n} = O(n)$ has two connected components, which in particular implies that a log-Sobolev inequality cannot hold.

\begin{proof}
	First recall that if $k < n$, we have
	\[
	W_{n,k} \cong SO(n)/SO(n-k).
	\]
	Indeed, identifying any matrix in $SO(n)$ with its first $n-1$ columns $e_1, \ldots, e_{n-1}$, this follows readily using the projection map $\varphi \colon SO(n) \to W_{n,k}$ which is given by $\varphi(e_1, \ldots, e_{n-1}) := (e_1, \ldots, e_k)$. By \cite[Th.\ 5.16]{Mec19}, the special orthogonal group $SO(n)$ equipped with the uniform probability measure and Hilbert--Schmidt metric satisfies a log-Sobolev inequality with constant $4/(n-2)$. Noting that the map $\varphi$ is $1$-Lipschitz, (1) therefore follows immediately.
\end{proof}

Similarly to the spherical case, the simplest method is now to work with Euclidean derivatives. Here, one notes that Proposition \ref{prop:LSUSt} naturally extends to the Euclidean gradient, giving back \cite[Th.\ 1.3]{GS23}.

\begin{proposition}\label{prop:HigherOrderSt}
Let $f$ be a $\mathcal{C}^d$-smooth function in a neighbourhood of $W_{n,k} \subset \R^{n \times k}$. Then, the results from Theorem \ref{thm:HigherOrderLSI} (adapted to the space $\R^{nk}$) remain valid for the uniform distribution $\mu_{n,k}$ with $\sigma^2 = 4/(n-2)$.
\end{proposition}

As in the case of the Euclidean sphere, at least a second order result in terms of intrinsic derivatives is possible. For details, cf.\ \cite[Th.\ 1.3]{GS23}.

\subsection{Grassmann manifolds}

Classically, together with Stiefel manifolds one also studies the closely related Grassmann manifold, i.\,e.\ the set of all $k$-dimensional subspaces of $\mathbb{R}^n$. There are several ways of introducing it (for an overview, cf.\ e.\,g.\ \cite{BZA24}). For instance, identifying a subspace with its basis (which is unique up to orthogonal transformations), we may understand the Grassmann manifold as the quotient $W_{n,k}/O(k)$. For our purposes, a different approach turns out to fit better.

For any $A \in W_{n,k}$, $\pist(A) := AA^T \equiv P_A$ is a projection matrix of rank $k$ (more precisely, the projection onto the subspace a basis of which is given by the columns of $A$), and for any $O_k \in O(k)$, we have $P_{AO_k} = P_A$. Therefore, identifying elements of the Grassmannian with projection matrices, we may define
\[
G_{n,k} := \{P \in \Rns \colon P^2 = P, \ \rank(P) = k \},
\]
where $\Rns$ denotes the space of the symmetric $n \times n$ matrices. $G_{n,k}$ is a manifold of dimension $\mathrm{dim}(G_{n,k}) = k(n-k)$. If $k=1$, we get back the half-sphere (where $\theta$ and $-\theta$ are identified), which can also be regarded as the projective space $\mathbb{R}P^{n-1}$. We can now either regard $G_{n,k}$ as a submanifold of $\Rns$ or $\mathbb{R}^{n \times n}$ (the latter is often more convenient when taking derivatives), equipping it with the inherited topology and distances similarly to the case of the Stiefel manifold (in particular, $G_{n,k}$ is compact).

Many of the properties of Stiefel manifolds can be extended to Grassmann manifolds. Especially, the group action $O(n) \ni O_n \mapsto P_{O_nA}$ turns $G_{n,k}$ into a symmetric (not just homogeneous) space. We denote the uniform distribution on $G_{n,k}$ by $\nu_{n,k}$. Clearly, $\nu_{n,k}$ is the pushforward of $\mu_{n,k}$ under the map $\pist$. If $G = (G_{ij})_{i,j=1}^{k,n}$ is an $n \times k$ random matrix whose entries are i.i.d.\ standard normal, it follows from the discussion above that $G(G^TG)^{-1}G^T \sim \nu_{n,k}$.

To introduce a notion of differentiability on $G_{n,k}$, first recall that the tangent space in $P \in G_{n,k}$ is given by
\[
T_P := \{S \in \Rns \colon S = SP + PS\} \equiv \{S \in \Rns \colon S = [[S,P],P]\},
\]
where for any $M,N \in \mathbb{R}^{n \times n}$, $[M,N] = MN - NM$ denotes the matrix commutator. Let us consider functions which are defined and smooth in an open neighbourhood of $G_{n,k}$ in the ambient space, e.\,g.\ $\Rns$ and take the Euclidean gradient $\nabla f(P) = \mathrm{mat}(Df(P)^T)$. Noting that the projection $\pi_P \colon \Rns \to T_P$ is given by
\begin{equation}\label{eq:projfGr}
	\pi_P M := [P,[P,M]] = PM + MP - 2PMP,
\end{equation}
in view of \eqref{eq:intgrad} the intrinsic derivative of $f$ at $P \in G_{n,k}$ is given by the projection onto $T_P$
\begin{equation}\label{eq:FormelGradientGrassmann}
\nabla_G f(P) = \pi_P\nabla f(P) = [P,[P,\nabla f(P)]].
\end{equation}

In particular, $|\nabla_G f(P)| \leq |\nabla f(P)|$ for any $P \in G_{n,k}$. Sometimes yet a further embedding might be convenient, so that we regard $G_{n,k}$ as a submanifold of $\mathbb{R}^{n \times n}$ and take the Euclidean derivatives of some extension of $f$ to an open neighbourhood in $\mathbb{R}^{n \times n}$. However, in this case, we may project $\nabla f(P)$ onto the tangent space of $\Rns$ (which equals $\Rns$) by applying the projection $\pis \colon \mathbb{R}^{n \times n} \to \Rns$ given by $\pis(M) := (M + M^T)/2$. Then, we may proceed as in \eqref{eq:projfGr} for $\pis(\nabla f(P))$.

Establishing a log-Sobolev inequality on $G_{n,k}$ can be performed by studying the Lipschitz properties of the map $\pist \colon W_{n,k} \to G_{n,k}$, $A \mapsto AA^T \equiv P_A$. Clearly, $\pist$ is $2\sqrt{k}$-Lipschitz as
\[
\lVert AA^T - A'A'^T \rVert_\mathrm{HS} \le \lVert A(A^T - A'^T) \rVert_\mathrm{HS} + \lVert (A-A')A'^T \rVert_\mathrm{HS} \le 2\sqrt{k} \lVert A-A' \rVert_\mathrm{HS}
\]
since $\lVert A \rVert_\mathrm{HS}, \lVert A' \rVert_\mathrm{HS} = \sqrt{k}$. This simple bound can however be significantly improved by arguments which are based on principle angles. Carrying out these arguments, it has been shown in \cite[Le.\ 6.1]{GS23} that $\pist$ is $\sqrt{2}$-Lipschitz.

\begin{proposition}[\cite{GS23}, Prop.\ 6.2]\label{prop:LSUGr}
For any $k < n$, $G_{n,k}$ satisfies a logarithmic Sobolev inequality with constant $8/(n-2)$, i.\,e.\ for any $f \colon G_{n,k} \to \mathbb{R}$ sufficiently smooth,
		\[
		\mathrm{Ent}_{\nu_{n,k}} (f^2) \le \frac{16}{n-2} \int_{G_{n,k}} |\nabla_G f|^2 d\nu_{n,k}.
		\]
\end{proposition}

\begin{proof}
	Note that $(G_{n,k}, \nu_{n,k})$ is the push-forward of $(W_{n,k},\mu_{n,k})$ under $\pist$, which is $\sqrt{2}$-Lipschitz. Therefore, the result follows immediately from Proposition \ref{prop:LSUSt}, doubling the Sobolev constant due to the Lipschitz property.
\end{proof}

Working with Euclidean derivatives, Theorem \ref{thm:HigherOrderLSI} thus extends to $G_{n,k}$, replacing $\sigma^2$ by $8/(n-2)$. For further details, see also \cite{GS23}, Th.\ 1.2 and Th.\ 1.4.

\section{Functions of independent and weakly dependent random variables}\label{sec:GenResDisc}

A common feature of all of the examples given in the previous section is the availability of some sort of differentiability, encoded by the generalized modulus of the gradient in whatever framework. The difference operator $\Gamma$ is then chosen accordingly. However, there are many situations of interest which do not belong to this class of examples. A case in point are discrete spaces, where the definition \eqref{eq:genmod} cannot be adapted.

Instead, we shall work with difference operators in the stricter sense of the word, i.\,e., certain operators which satisfy \eqref{eq:DiffOp} and typically involve differences of the functional $f$ under consideration. It turns out that many of the arguments from the previous sections can be adapted by choosing the right difference operator. However, the correct choice of the difference operator depends on the argument one wishes to adapt, and consequently, it often happens to be necessary to work with more than one difference operator. Moreover, unlike in the case of generalized moduli of gradients, some arguments only work for the positive part of the functions involved.

In view of these points, we need a generalized variant of Assumptions \ref{ass:Set1} and hence introduce the following framework.

\begin{assumptions}\label{ass:Set2}
	\begin{enumerate}
		\item There exist $p > 0$, $r_0 > 1$, and $L, \sigma > 0$ such that for any function $g \in \mathcal{F}$, we have
		\begin{equation*}
			\lVert g - \E g \rVert_r \le L\sigma r^{1/p} \lVert \Gamma g \rVert_r
		\end{equation*}
		for all $r \ge r_0$.
		\item For the same quantities $p > 0$, $r_0 > 1$ and $L, \sigma > 0$ as in (1) and any function $g \in \mathcal{F}$, we have
		\begin{equation*}
			\lVert (g - \E g)_+ \rVert_r \le L\sigma r^{1/p} \lVert \Gamma^+ g \rVert_r
		\end{equation*}
		\item There exist ``higher order'' difference operators $\Gamma^{(j)}(g) \in [0,\infty)$, $j = 1, \ldots, d$, such that for any $g \in \mathcal{F}$,
		\begin{equation*}
			\Gamma^+ (\Gamma^{(j)}(g)) \le \Gamma^{(j+1)}(g)
		\end{equation*}
		for any $j = 1, \ldots, d-1$.
		\item[(3')] There exist ``higher order'' difference operators $\Gamma^{(j)}(g) \in [0,\infty)$, $j = 1, \ldots, d$, and some constant $\gamma > 0$ such that for any $g \in \mathcal{F}$, $\Gamma^+(g) \le \gamma \Gamma(g)$ as well as
		\begin{equation*}
			\Gamma^+ (\Gamma^{(j)}(g)) \le \gamma \Gamma^{(j+1)}(g)
		\end{equation*}
		for any $j = 1, \ldots, d-1$.
	\end{enumerate}
\end{assumptions}

To get back Assumptions \ref{ass:Set1}, one just has to set $\Gamma = \Gamma^+$, noting that in this case Assumption \ref{ass:Set2} (2) follows from Assumption \ref{ass:Set2} (1). Assumption \ref{ass:Set2} (3') should be regarded as an alternate and slightly even more general version of Assumption \ref{ass:Set2} (3). Other modifications are possible as well, but in view of our applications we have chosen this way of presentation.

\begin{satz}\label{thm:HigherOrder2}
	Given Assumptions \ref{ass:Set2} (1)--(3), let $f\colon \Omega \to \R$ be a suitable function with $\E f = 0$.
	\begin{enumerate}
		\item Assuming that $\lVert \Gamma^{(j)} f \rVert_1 \le \sigma^{d-j}$ for all $j=1,\ldots,d-1$ and $\lVert \Gamma^{(d)} f \rVert_\infty \le 1$, it holds that
		\[
		\int \exp\Big\{\frac{c_{p,d,r_0,L}}{\sigma^p} |f|^{p/d} \Big\} d\mu \le 2,
		\]
		where a possible choice of the constant $c_{p,d,r_0,L}$ is given by
		\[
		c_{p,d,r_0,L} = \frac{(r_0^{1/p}-1)^p}{2e\max(L^{1/d}, L)^pr_0\max(r_0,p/d)}.
		\]
		\item For any $t\ge 0$, it holds that
		\[
		\mu(|f| \ge t)\le
		2 \exp\Big\{- \frac{C_{p,r_0,L}}{(d\sigma)^p} \min\Big(\min_{k=1, \ldots,d-1} \Big(\frac{t}{\lVert \Gamma^{(k)} f \rVert_1}\Big)^{p/k}, \Big(\frac{t}{\lVert \Gamma^{(d)} f \rVert_\infty}\Big)^{p/d}\Big)\Big\},
		\]
		where a possible choice of the constant $C_{p,r_0,L}$ is given by
		\[
	    C_{p,r_0,L} = \frac{\log(2)}{r_0(Le)^p}.
		\]
	\end{enumerate}
\end{satz}

Variants of Theorem \ref{thm:HigherOrder2} include the following situations.

\begin{proposition}\label{prop:HigherOrder3}
Consider the framework of Theorem \ref{thm:HigherOrder2}.
\begin{enumerate}
\item If only Assumptions \ref{ass:Set2} (2)\&(3) are satisfied, the same bounds as in Theorem \ref{thm:HigherOrder2} continue to hold with $|f|$ replaced by $f_+$. In particular, (2) is an inequality for the upper tails $\mu(f\ge t)$ in this case.
\item The results from Theorem \ref{thm:HigherOrder2} remain valid if (3) is replaced by (3'), with the only modification that both $c_{p,d,r_0,L}$ and $C_{p,r_0,L}$ must be multiplied by $1/\gamma^p$ in this case.
\end{enumerate}
\end{proposition}

Note that in all the examples in Section \ref{sec:Ex2}, we have $p=2$, so that in principle, Assumptions \ref{ass:Set2} and Theorem \ref{thm:HigherOrder2} could be formulated in this situation only. In the context of this note, we have chosen to stick to the more general framework of arbitrary $p > 0$.

\begin{proof}[Proof of Theorem \ref{thm:HigherOrder2}]
	The proof is very similar to the proof of Theorem \ref{thm:HigherOrder}. First, Assumption \ref{ass:Set2} (1) yields
	\begin{align*}
		\norm{f}_r \le L\sigma r^{1/p} \norm{\Gamma f}_p \le L\sigma r^{1/p} \norm{\Gamma f}_1 + L\sigma r^{1/p} \norm{(\Gamma f - \E \Gamma f)_+}_r
	\end{align*}
	where we have used that for any positive random variable $W$
	\begin{align}		\label{eqn:LpEstimatePositiveFunction}
		\norm{W}_r \le \IE W + \norm{(W - \IE W)_+}_r.
	\end{align}
	Now, by Assumptions \ref{ass:Set2} (2)\&(3) it follows that
	\begin{align*}
		\norm{\left( \Gamma f - \E \Gamma f \right)_+}_r \le L\sigma r^{1/p} \norm{\Gamma^+ (\Gamma f)}_r \le L\sigma r^{1/p} \norm{\Gamma^{(2)}f}_r.
	\end{align*}
	This can be easily iterated to obtain for any $d \in \IN$
	\begin{align*}
		\norm{f}_r \le \sum_{j = 1}^{d-1} (L\sigma r^{1/p})^j \norm{\Gamma^{(j)}f}_1 + (L\sigma r^{1/p})^d \norm{\Gamma^{(d)}f}_\infty.
	\end{align*}
	If (1) is not satisfied, then one uses (2) in the first step instead, which leads to
	\begin{align*}
		\norm{f_+}_r \le \sum_{j = 1}^{d-1} (L\sigma r^{1/p})^j \norm{\Gamma^{(j)}f}_1 + (L\sigma r^{1/p})^d \norm{\Gamma^{(d)}f}_\infty.
	\end{align*}
	Now one can argue as in the proof of Theorem \ref{thm:HigherOrder}. The case of (3) being replaced by (3') follows in the same way.
\end{proof}

\section{Examples}\label{sec:Ex2}

\subsection{Independent random variables}

In the sequel, we will consider sets of random variables $X=(X_1, \ldots, X_n)$ (with values in some measurable space $(\mathcal{X}, \mathcal{B}) = \otimes_{i =1}^n (\mathcal{X}_i, \mathcal{B}_i)$) and functions $f(X) = f(X_1, \ldots, X_n)$ (in particular, this means a slight switch of notation with $\mu$ from above being the distribution of $X$, however all quantities can be readily adapted). In this section, we will moreover assume the variables $X_1, \ldots, X_n$ to be independent. To formulate the corresponding results, it is sufficient to consider what we may call $L^\infty$ difference operators, which frequently appear in the method of bounded differences.

Let $X' = (X_1',\ldots, X'_n)$ be an independent copy of $X$, defined on the same probability space. Given $f(X) \in L^\infty(\mathbb{P})$, define for each $i = 1, \ldots, n$
\begin{equation*}
	T_i f \coloneqq T_if(X) \coloneqq f(X_{i^c}, X_i') = f(X_1, \ldots, X_{i-1}, X_i',\linebreak[2] X_{i+1},\ldots, X_n)
\end{equation*}
and
\begin{equation}	\label{eq:Operatorh}
	\mathfrak{h}_i f(X) = \lVert f(X) - T_if(X) \rVert_{i, \infty},
	\qquad \mathfrak{h} f(X) = (\mathfrak{h}_1 f(X), \ldots, \mathfrak{h}_n f(X)),
\end{equation}
where $\lVert \cdot \rVert_{i, \infty}$ denotes the $L^\infty$-norm with respect to $(X_i,X_i')$. The difference operator $\abs{\mathfrak{h}f}$ is then given as the Euclidean norm of the vector $\mathfrak{h} f$.

By establishing Assumptions \ref{ass:Set2}, we may show the following result, whose second part mirrors \cite[Th.\ 1]{GSS21a}.

\begin{satz}\label{thm:HigherOrderIndep}
    Let $X$ be a random vector with independent components, $f: \mathcal{X} \to \R$ a measurable function satisfying $f = f(X) \in L^\infty(\IP)$, $\E f(X) = 0$ and $d \in \N$.
    \begin{enumerate}
        \item Assuming that $\norm{\mathfrak{h}^{(j)} f}_{\mathrm{op},1} \le 1$ for all $j=1,\ldots,d-1$ and $\norm{\mathfrak{h}^{(d)} f}_{\mathrm{op},\infty} \le 1$, it holds that
        \[
        \E \exp\Big\{c |f(X)|^{2/d} \Big\} \le 2,
        \]
        where a possible choice of the constant $c$ is given by
        \[
        c = \frac{(\sqrt{2}-1)^2}{64 \kappa e}.
        \]
        \item For any $t\ge 0$, it holds that
        \[
        \P(|f(X)| \ge t)\le
        2 \exp\Big\{- \frac{C}{d^2} \min\Big(\min_{j=1, \ldots,d-1} \Big(\frac{t}{\norm{\mathfrak{h}^{(j)} f}_{\mathrm{op},1}}\Big)^{2/j}, \Big(\frac{t}{\norm{\mathfrak{h}^{(d)} f}_{\mathrm{op},\infty}}\Big)^{2/d}\Big)\Big\},
        \]
    where a possible choice of the constant $C$ is given by
    \[
    C = \frac{\log(2)}{16 \kappa e^2} > \frac{1}{217}.
    \]
    \end{enumerate}
\end{satz}

Let us provide a brief sketch of the main arguments leading to Theorem \ref{thm:HigherOrderIndep}. In addition to $|\mathfrak{h}|$ from \eqref{eq:Operatorh}, we also introduce two closely related difference operators. For $i \in \{1, \ldots, n\}$, let
\begin{align*} 
	&\mathfrak{h}^+_if(X) = \lVert (f(X) - T_if(X))_+ \rVert_{X_i', \infty},
	\qquad \mathfrak{h}^+f = (\mathfrak{h}^+_1f, \ldots, \mathfrak{h}^+_nf),\\
	&\mathfrak{h}^-_if(X) = \lVert (f(X) - T_if(X))_- \rVert_{X_i', \infty},
	\qquad \mathfrak{h}^-f = (\mathfrak{h}^-_1f, \ldots, \mathfrak{h}^-_nf),
\end{align*}
where $\norm{f}_{X_i',\infty}$ shall denote the $L^\infty$ norm with respect to $X_i'$. As usual, one now defines $|\mathfrak{h}^+ f|$ and $|\mathfrak{h}^- f|$ to be the respective Euclidean norms. In the framework of Assumptions \ref{ass:Set2}, one has $\Gamma (f) = |\mathfrak{h} f|$ and $\Gamma^+ (f) = |\mathfrak{h}^+ f|$, while $|\mathfrak{h}^- f|$ can be regarded as an auxiliary difference operator.

To verify Assumptions \ref{ass:Set2} (1)\&(2), we establish the following theorem.

\begin{satz}\label{th:BBLM}
	If $X_1, \ldots, X_n$ are independent random variables and $f = f(X) \in L^\infty(\P)$, with the constant $\kappa = \frac{\sqrt{e}}{2\,(\sqrt{e} - 1)}$, we have for any $r \ge 2$,
	\begin{equation*} 
		\lVert (f - \mathbb{E}f)_+ \rVert_r \le (2\kappa r)^{1/2}\, \lVert \mathfrak{h}^+f \rVert_r \quad \text{and} \quad
		\lVert (f - \mathbb{E}f)_- \rVert_r \le (2\kappa r)^{1/2}\, \lVert \mathfrak{h}^-f \rVert_r.
	\end{equation*}
	Consequently, this leads to 
	\[
	\norm{f - \IE f}_r \le (8\kappa r)^{1/2} \norm{\mathfrak{h}f}_r.
	\]
\end{satz}

In fact, replacing $\mathfrak{h}^\pm$ by (the square roots of) the closely related quantities $V^\pm(f) = \E(\sum_{i = 1}^n (f-T_if)^2_\pm|X)$ (which can be regarded as $L^2$-type difference operators), the first line of Theorem \ref{th:BBLM} is actually the content of \cite[Th.\ 2]{BBLM05}.

For a proof, one may follow the arguments from \cite{BBLM05} and those in \cite[Sec.\ 2]{BGS19} (which restate the lines of the proof of \cite{BBLM05} in a self-contained, slightly adapted and somewhat simplified manner). To mention some of the key ingredients, in essence the proof follows by some sort of induction which consecutively addresses $r \in (k,k+1]$ for $k \in \N$, where the first step follows from the Efron--Stein inequality.

A central argument is some kind of moment recursion which by tensorization arguments can be reduced to the case of $n=1$. Here, in the language of \cite{BBLM05}, one uses so-called (modified) $\Phi$-Sobolev inequalities. These are generalizations of log-Sobolev inequalities which involve $\Phi$-entropies, where $\Phi:\R_+\to\R$ is some function which satisfies the following three properties:
\begin{itemize}
    \item[(i)] $\Phi$ is convex and continuous,
    \item[(ii)] $\Phi$ is twice differentiable on $(0,\infty)$,
    \item[(iii)] $\Phi$ is either affine, or $\Phi''$ is strictly positive and $1/\Phi''$ is concave.
\end{itemize}
Typical examples are $\Phi_{\log}(x)=x\log x$ or $\Phi_q(x)=x^{2/q}$ with $q\in(1,2]$ and $x\in\R_+$. The $\Phi$-entropy of a random variable $f$ is defined as
$$
\Ent_\Phi(f) := \E[\Phi(f)] - \Phi(\E[f]).
$$
In particular, the classical entropy $\Ent(f)=\E[f\log f]-\E[f]\log(\E[f])$ of $f$ is recovered by taking $\Phi=\Phi_{\log}$.

The functions relevant for the arguments from \cite{BBLM05} are the power functions $\Phi_q(x)$ for $q \in (1,2]$. In this context, a fact which turns out to be important is that the tensorization property
	\begin{equation}
		\label{eq:Tensorisierung}
		\mathbb{E}\, |f|^q - \big(\mathbb{E}\, |f|\big)^q \le 
		\mathbb{E}\, \sum_{i=1}^{n} \Big(\mathbb{E}_i\, |f|^q - 
		\big(\mathbb{E}_i\, |f|\big)^q\Big),
	\end{equation}
holds for any $f \in L^q$ and any $q \in (1,2]$, where $\mathbb{E}_i$ denotes expectation with respect to $X_i$. In fact, tensorization holds for $\Phi$-entropies in general, cf.\ \cite[Prop.\ 4.1]{Led96} and also \cite{LO00}.

To establish Assumption \ref{ass:Set2} (3), we need higher order versions of $\mathfrak{h}$, denoted by $\mathfrak{h}^{(j)} f$. They can be thought of as analogues of the $j$-tensors of all partial derivatives of order $j$ in an abstract setting. To define the $j$-tensor $\mathfrak{h}^{(j)} f$, we specify it on its ``coordinates''. That is, given distinct indices $i_1, \ldots, i_j$, we set
\begin{align} 		\label{eq:DefiHigherOrderDerivatives}
	\begin{split}
		\mathfrak{h}_{i_1 \ldots i_j}f(X) 
		= \; &
		\Big\lVert\, \prod_{s=1}^j\, 
		(\mathrm{Id} - T_{i_s})f(X) \Big\rVert_{i_1, \ldots, i_j, \infty}\\
		= \; &
		\Big\lVert\, f(X) + \sum_{\ell=1}^j\, (-1)^\ell \sum_{1 \leq s_1 < \ldots < s_\ell \leq j}
		T_{i_{s_1} \ldots i_{s_\ell}}f(X)\, \Big\rVert_{i_1, \ldots, i_j, \infty}
	\end{split}
\end{align}
where $T_{i_1 \ldots i_j} = T_{i_1} \circ \ldots \circ T_{i_j}$ exchanges the random variables $X_{i_1}, \ldots, X_{i_j}$ by $X'_{i_1}, \ldots, X'_{i_j}$, and 
$\lVert \cdot \rVert_{i_1, \ldots, i_j, \infty}$ denotes the $L^\infty$-norm with respect to the random variables $X_{i_1}, \ldots, X_{i_j}$ and $X_{i_1}', \ldots, X_{i_j}'$. For instance, for $i \ne j$,  
\[
\mathfrak{h}_{ij} f(X) = \lVert f(X) - T_if(X) - T_jf(X) + T_{ij}f(X) \rVert_{i,j,\infty}.
\]
Using the definition \eqref{eq:DefiHigherOrderDerivatives}, we define tensors of $j$-th order differences as follows:
\begin{align*} 
	\big(\mathfrak{h}^{(j)}f(X)\big)_{i_1 \ldots i_j} = 
	\begin{cases} 
		\mathfrak{h}_{i_1 \ldots i_j}f(X), & 
		\text{if $i_1, \ldots, i_j$ are distinct}, \\ 0, & \text{else}. 
	\end{cases}
\end{align*}
Typically, we omit writing the random vector $X$, i.\,e. we write $f$ instead of $f(X)$ and $\mathfrak{h}^{(j)}f$ instead of $\mathfrak{h}^{(j)}f(X)$.

Note that $\abs{\mathfrak{h}^{(j)} f}_\mathrm{HS}$ and $\abs{\mathfrak{h}^{(j)} f}_\mathrm{op}$ are again difference operators. We are now ready to illustrate how and why Assumptions \ref{ass:Set2} differ from Assumptions \ref{ass:Set1}. First, by \cite[Le.\ 2.2]{BGS19}, we have that
\begin{equation}
	\label{eq:Rekursion}
	|\mathfrak{h}|\mathfrak{h}^{(j)}f(X)|_\mathrm{HS}| \le |\mathfrak{h}^{(j+1)}f(X)|_\mathrm{HS},
\end{equation}
which establishes Assumption \ref{ass:Set1} (2) for Hilbert--Schmidt norms. The arguments leading to \eqref{eq:Rekursion} are simple (essentially, everything follows by triangle inequality), however, they cannot be adapted to operator norms. In fact, an analogue of \eqref{eq:Rekursion} for operator norms does not hold true, which can already be seen in second order statistics of Rademacher variables like $\sum_{i=1}^{n} X_iX_{i+1}$ (setting $X_{n+1} = X_1$).

The situation changes if we also work with $\mathfrak{h}^+$. The crucial step is the following observation: if $A$ is a $j$-tensor, the supremum in the definition of $\abs{A}_\mathrm{op}$ is attained, and if $A$ is a non-negative tensor (i.\,e. $A_{i_1 \ldots i_j} \ge 0$ for all $i_1, \ldots, i_j$), the maximizing vectors $\tilde{v}^1, \ldots, \tilde{v}^j$ can be chosen to have all positive entries. Indeed, since $\tilde{v}^1_{i_1} \cdots \tilde{v}^k_{i_j} \le \abs{\tilde{v}^1_{i_1} \cdots \tilde{v}^k_{i_j}}$, we can define $\abs{\tilde{v}}^\ell$ by taking the absolute value elementwise.

\begin{lemma}[\cite{GSS21a}, Le.\ 1]		\label{lem:recursiveEstimateOperatorNorms}
	For any $j \in \N$,
	\[
	\abs{\mathfrak{h}^+ \abs{\mathfrak{h}^{(j)} f(X)}_\mathrm{op}} \le \abs{\mathfrak{h}^{(j+1)} f(X)}_\mathrm{op}.
	\]
\end{lemma}

\begin{proof}
	We have
	\begin{align*}
		&\abs{\mathfrak{h}^+ \abs{\mathfrak{h}^{(j)} f}_\mathrm{op}}^2 = \sum_{i= 1}^n \norm*{\left( \abs{\mathfrak{h}^{(j)} f}_\mathrm{op} - \abs{\mathfrak{h}^{(j)} T_i f}_\mathrm{op} \right)_+}_{i,\infty}^2 \\
		&= \sum_{i=1}^n \bnorm{\left( \sup_{v^\ell} \skal{v^1 \cdots v^j, \mathfrak{h}^{(j)}f} - \sup_{v^\ell} \skal{v^1 \cdots v^j, \mathfrak{h}^{(j)}T_i f} \right)_+}_{i,\infty}^2 \\
		&\le \sum_{i = 1}^n \bnorm{\bigg(\skal{\tilde{v}^1\cdots \tilde{v}^j, \mathfrak{h}^{(j)}f - \mathfrak{h}^{(j)} T_i f} \bigg)_+}_{i,\infty}^2 \\
		&\le \sum_{i = 1}^n \bnorm{\sum_{i_1,\ldots,i_j} \tilde{v}^1_{i_1} \cdots \tilde{v}^j_{i_j} \bnorm{(\Id - T_i) \prod_{\ell = 1}^j (\Id - T_{i_\ell})f}_{i_1 \cdots i_j,\infty} }_{i,\infty}^2 \\
		&\le \sum_{i =1}^n \bigg( \sum_{i_1, \ldots, i_j} \tilde{v}^1_{i_1} \cdots \tilde{v}^j_{i_j} \mathfrak{h}_{i i_1 \cdots i_j}f \bigg)^2 \\
		&= \bigg( \sup_{v^{j+1} : \abs{v^{j+1}} \le 1} \sum_{i_{j+1}=1}^n \sum_{i_1, \ldots, i_j} \tilde{v}_{i_1}^1 \cdots \tilde{v}^j_{i_j} v^{j+1}_{i_{j+1}} \mathfrak{h}_{i_1 \cdots i_{j+1}} f \bigg)^2 \\
		&\le \bigg( \sup_{v^1, \ldots, v^{j+1} : \abs{v^\ell} \le 1} \sum_{i_1, \ldots, i_{j+1}} v_{i_1}^1 \cdots v_{i_{j+1}}^{j+1} \mathfrak{h}_{i_1 \cdots i_{j+1}}f \bigg)^2 \\
		&= \abs{\mathfrak{h}^{(j+1)}f}_\mathrm{op}^2
	\end{align*}
	Here, in the first inequality we insert the vectors $\tilde{v}^1, \ldots, \tilde{v}^j$ maximizing the supremum and use the monotonicity of $x \mapsto x_+$, and the second and third inequality follow from the triangle inequality. Taking the square root yields the claim.
\end{proof}

\subsection{Weakly dependent random variables}

For results beyond independence, the dependencies clearly cannot be arbitrary, but we have to establish concepts of ``weak'' dependence, i.\,e., situations in which the dependencies can still be controlled and the independent case can be ``imitated'' in whatever way. The central idea is to work with logarithmic Sobolev inequalities again, but now adapted to discrete-type situations (cf.\ e.\,g.\ \cite{DSC96,BT06}).

As we will see in detail slightly later, the difference operator $|\mathfrak{h}|$ from the previous section is not a suitable choice for the entropy method however. Instead, we will work with an $L^2$-difference operator, defined as follows. Assume that $\mathcal{X} = \otimes_{i = 1}^n \mathcal{X}_i$ for some finite sets $\mathcal{X}_1, \ldots, \mathcal{X}_n$, equipped with a probability measure $\mu$. For a vector $x = (x_1, \ldots, x_n)$ and $I \subset \{1,\ldots,n\}$ we write $x_I \coloneqq (x_i)_{i \in I}$ and $x_{I^c} \coloneqq (x_i)_{i \notin I}$. If $I = \{ j\}$ for some $j \in \{1,\ldots,n\}$ we abbreviate it as $x_j$ and $x_{j^c}$. Let $\mu(\cdot \mid x_{i^c})$ denote the conditional measure (interpreted as a measure on $\mathcal{X}_i$) and $\mu_{i^c}$ the marginal on $\otimes_{j \ne i} \mathcal{X}_j$. Finally, set
\begin{align*}
	\abs{\dpartial f}^2(x) &\coloneqq \sum_{i = 1}^n (\dpartial_i f(x))^2 \coloneqq \sum_{i = 1}^n \mathrm{Var}_{\mu(\cdot \mid x_{i^c})}(f(x_{i^c}, \cdot)) \\
	&=\sum_{i = 1}^n \frac{1}{2} \iint (f(x_{i^c},y) - f(x_{i^c},y'))^2 d\mu(y \mid x_{i^c}) d\mu(y' \mid x_{i^c}).
\end{align*}
Roughly speaking, $|\dpartial f|$ is based on a sort of ``conditional variances'' in the respective components. It frequently appears in the context of logarithmic Sobolev inequalities for Markov chains as developed in \cite{DSC96}.
This difference operator appears naturally in the Dirichlet form associated to the Glauber dynamic of $\mu$, given by
\begin{align*}	
	\mathcal{E}(f,f) \coloneqq \sum_{i = 1}^n \int \Var_{\mu(\cdot \mid x_{i^c})}(f(x_{i^c},\cdot)) d\mu_{i^c}(x_{i^c}) = \int \abs{\dpartial f}^2 d\mu.
\end{align*}
In the main theorem of this section, we require a $\dpartial$--LSI for the underlying probability measure $\mu$. A number of models which satisfy this assumption will be discussed below.

Before we discuss $\dpartial$-LSIs in more detail, let us first see how Assumptions \ref{ass:Set2} are verified in this situation. The main tool is the following result.

\begin{proposition}[\cite{GSS21a}, Prop.\ 5]	\label{prop:momentinequality}
	Let $\mu$ be a measure on a product of finite sets satisfying a $\dpartial\mathrm{-LSI}(\sigma^2)$. Then, for any $f \in L^\infty(\mu)$ and any $r \ge 2$ we have
	\begin{equation} \label{eqn:fMinusEfLP}
		\norm{f - \IE f}_r \le (\sigma^2(r-3/2)/2)^{1/2} \norm{\mathfrak{h} f}_r
	\end{equation}
	and
	\begin{equation}\label{eqn:fMinusEfLP2}
		\norm{(f - \IE f)_+}_r \le (2\sigma^2 (r-3/2))^{1/2} \norm{\mathfrak{h}^+ f}_r.
	\end{equation}
\end{proposition}

To establish Proposition \ref{prop:momentinequality}, one mimics the proof of \eqref{eq:AS}. The lack of the chain rule inequality in discrete settings can be compensated by invoking the pointwise inequality $(a^{r/2} - b^{r/2})_+^2 \le \frac{r^2}{4} a^{r-2}(a-b)_+^2$, valid for all $a,b \ge 0$ and all $r \ge 2$.

Proposition \ref{prop:momentinequality} hence establishes Assumptions \ref{ass:Set2} (1)\&(2), while (3) is satisfied by means of Lemma \ref{lem:recursiveEstimateOperatorNorms} once again. In particular, we immediately arrive at the following analogue of Theorem \ref{th:BBLM}, whose second part mirrors [\cite{GSS21a}, Th.\ 2].

\begin{satz}		\label{thm:partialLSITails}
	Let $\mu$ be a measure on a product of finite sets satisfying a $\dpartial\mathrm{-LSI}(\sigma^2)$, and let $f \in L^\infty(\mu)$.
	\begin{enumerate}
	        \item Assuming that $\norm{\mathfrak{h}^{(j)} f}_{\mathrm{op},1} \le 1$ for all $j=1,\ldots,d-1$ and $\norm{\mathfrak{h}^{(d)} f}_{\mathrm{op},\infty} \le 1$, it holds that
	        \[
	        \E \exp\Big\{c |f(X)|^{2/d} \Big\} \le 2,
	        \]
	        where a possible choice of the constant $c$ is given by
	        \[
	        c = \frac{(\sqrt{2}-1)^2}{16 \kappa e}.
	        \]
	        \item For any $t\ge 0$, it holds that
	        \[
	        \P(|f(X)| \ge t)\le
	        2 \exp\Big\{- \frac{C}{d^2} \min\Big(\min_{j=1, \ldots,d-1} \Big(\frac{t}{\norm{\mathfrak{h}^{(j)} f}_{\mathrm{op},1}}\Big)^{2/j}, \Big(\frac{t}{\norm{\mathfrak{h}^{(d)} f}_{\mathrm{op},\infty}}\Big)^{2/d}\Big)\Big\},
	        \]
	        where a possible choice of the constant $C$ is given by
	    	\[
	    	C = \frac{\log(2)}{4 \kappa e^2} > \frac{1}{34}.
	    	\]
	 \end{enumerate}
\end{satz}

In principle, for $d=2$ we could now derive Hanson--Wright type inequalities again. Starting with a quadratic form involving some matrix $A=(a_{ij})_{ij}$, the difference operator $\mathfrak{h}$ would lead to bounds which involve the matrix $A^\mathrm{abs}$, defined by $A^\mathrm{abs}_{ij} = |a_{ij}|$ (cf.\ the discussion following \cite[Th.\ 1]{GSS21a}). A somewhat more general framework in which this slight pathology does not appear will be discussed in Section \ref{sec:PolCh}.

\subsubsection{Logarithmic Sobolev inequalities involving difference operators}

Let us discuss the $\mathfrak{d}$--LSI condition in more detail. In typical situations, the following observation plays a central role. Recall that for product measures, establishing a log-Sobolev-type inequality can be reduced to dimension $1$ by means of tensorization of the entropy functional as seen in a generalized sense in \eqref{eq:Tensorisierung}. The latter makes use of independence. Nevertheless, it turns out one can hope for ``approximate'' analogues even without independence, more precisely, inequalities of the form
\begin{equation}\label{eq:AT}
	\Ent_{\mu}(f^2) \le C \sum_{i = 1}^n \int \Ent_{\mu(\cdot \mid x_{i^c})}(f^2(x_{i^c}, \cdot)) d\mu_{i^c}(x_{i^c}),
\end{equation}
which is called approximate tensorization with constant $C$ or $\mathrm{AT}(C)$ in short. Once such an inequality is valid, one is again able to reduce to dimension $1$.

For a probability measure $\mu$ on a product of finite sets $\mathcal{X} = \otimes_{i = 1}^n \mathcal{X}_i$, any matrix $J = (J_{ij})_{i,j}$ with $J_{ii} = 0$ for all $i=1, \ldots, n$ and such that for any $x,y \in \mathcal{X}$ with $x_{j^c} = y_{j^c}$ it holds that
\[
d_\mathrm{TV}(\mu(\cdot|x_{i^c}), \mu(\cdot|y_{i^c})) \le J_{ij},
\]
where $d_\mathrm{TV}$ is the total variation distance. The matrix $J$ is called an \emph{interdependence matrix}. It measures the strength of the interactions between the random variables. In particular, if they are independent (equivalently, if $\mu$ is a product measure), we may set $J \equiv 0$.

Furthermore, define for any subset $S \subsetneq \{1, \ldots, n\}$ and any $i \notin S$ (writing $\mathcal{X}_S := \otimes_{i \in S} \mathcal{X}_i$ and $\mathcal{X}_{S^c}$ accordingly)
\begin{equation}\label{eqn:DefiBetaTilde}
  \tilde{\beta}_{i,S}(\mu) \coloneqq \inf_{\substack{x_S \in \mathcal{X}_S \\ \mu_S(x_S) > 0}} \inf_{\substack{y_{S^c} \in \mathcal{X}_{S^c} \\ \mu(y_{S^c}, x_{S}) > 0} }  \mu ( (y_{S^c})_i \mid x_S).
\end{equation}
If $S = \emptyset$, this reads $\tilde{\beta}_{i,\emptyset}(\mu) = \inf_{x \in \mathcal{X} : \mu(x) > 0} \mu(x_i).$ The interpretation of $\tilde{\beta}_{i,S}(\mu)$ is straightforward: For any admissible partial configuration $x_S \in \mathcal{X}_S$ all possible marginals are supported on points with probability at least $\tilde{\beta}_{i,S}(\mu)$. Now let
\begin{equation*} 
  \tilde{\beta}(\mu) \coloneqq \inf_{S \subsetneq \mathcal{I}} \inf_{i \notin S} \tilde{\beta}_{i,S}(\mu)
\end{equation*}
be the infimum of all $\tilde{\beta}_{i,S}(\mu)$. If $\mu$ is a probability measure, $\tilde{\beta}(\mu)$ reduces to the minimal probability of any atom in the marginal distributions.

We now have the following result, going back to \cite{Mar19} as well as \cite[Th.\ 4.2]{GSS19}.

\begin{proposition}\label{prop:AT+LSI}
Let $\mu$ be a probability measure on a product of of finite sets $\mathcal{X} = \otimes_{i = 1}^n \mathcal{X}_i$. Assume that there are $\alpha_1, \alpha_2 \in (0,1)$ such that $\tilde{\beta}(\mu) \ge \alpha_1$ and that there is some interdependence matrix $J$ such that $\lvert J \rvert_\mathrm{op} \le 1 - \alpha_2$.
\begin{enumerate}
\item The approximate tensorization property \eqref{eq:AT} holds true with $C= 1/(\alpha_1\alpha_2^2)$.
\item $\mu$ satisfies a $\dpartial$-LSI with constant $\sigma^2 = \log(\alpha_1^{-1})/(2\log(2)\alpha_1\alpha_2^2)$.
\end{enumerate}
\end{proposition}

The condition $\lvert J \rvert_\mathrm{op} \le 1 - \alpha_2$ is also called the Dobrushin condition.

\begin{bemerkung}
If $\mu$ is a product measure, we may set $\alpha_2=1$. However, in any non-trivial example, we clearly do not have $\alpha_1=1$. On the other hand, for product measures, \eqref{eq:AT} always holds with $C=1$, which thus cannot be recovered from Proposition \ref{prop:AT+LSI}. The factor $\alpha_1$ cannot be avoided by the proof strategy going back to \cite{Mar19}. It is of course tempting to ask whether it is an artefact of the proof.
\end{bemerkung}

Proposition \ref{prop:AT+LSI} and all the quantities appearing therein can easily be generalized to products $\mathcal{X}^\mathcal{I}$, where $\mathcal{X}$ and $\mathcal{I}$ are finite spaces. Probability measures $\mu$ on such a space $\mathcal{X}^\mathcal{I}$ are called spin systems. Spin systems for which $\dpartial$-LSIs have been shown include the following models:
\begin{itemize}
\item the Ising model (\cite{BGS19}),
\item the exponential random graph model (\cite{SS20}),
\item the random colouring model (\cite{SS20}),
\item the hard-core model (\cite{SS20}),
\item the (block) Potts model (\cite{KLS21}).
\end{itemize}

Moreover, there are various results on the interplay between difference operators and log-Sobolev inequalities. Let us provide a brief summary.
\begin{itemize}
\item In principle, the difference operator $|\dpartial f|$ can be extended to products of Polish spaces by means of the disintegration theorem. However, it turns out that the $\dpartial$-LSI property in fact requires the underlying space to be finite. See \cite[Prop.\ 6]{GSS21a}.
\item Any product measure satisfies an $\mathfrak{h}$-LSI with universal constant $\sigma^2=1$. More generally, approximate tensorization \eqref{eq:AT} with constant $C$ implies an $\mathfrak{h}$-LSI with the same constant $C$. However, there is no way to derive concentration results from these facts by means of the entropy method. See \cite[Th.\ 7]{GSS21a} and the discussion thereafter.
\end{itemize}

\begin{bemerkung}
	There is yet a different (but still unexplored) approach for proving (higher order) concentration results without independence, namely in terms of approximate tensorization of $\Phi$-entropies. Indeed, revisiting the proof of Theorem \ref{th:BBLM} once again, all the arguments remain valid without assuming independence if an analogue of \eqref{eq:Tensorisierung} holds. More precisely, one has to require that
	\[
	\mathbb{E}\, |g|^q - \big(\mathbb{E}\, |g|\big)^q \le C
	\mathbb{E}\, \sum_{i=1}^{n} \Big(\mathbb{E}_i\, |g|^q - 
	\big(\mathbb{E}_i\, |g|\big)^q\Big)
	\]
	for all $q \in (1,2]$ with some constant $C$ which holds uniformly in $q$. For independent random variables, one has $C=1$ as recalled in the course of the proof. Whether such a property can be verified in any other situation does not seem clear at present, but if so everything else will adapt to this situation and one obtains a class of new concentration results.
\end{bemerkung}

\subsection{Suprema of polynomial chaos}\label{sec:PolCh}

Let us briefly provide a final example where the variants presented in Proposition \ref{prop:HigherOrder3} are needed. We shall derive uniform concentration bounds for families of polynomials.

In detail, we will focus on the following setup. Let $\mathcal{I}_{n,d}$ denote the family of subsets of $\{1, \ldots, n\}$ with $d$ elements, fix a real Banach space $(\mathbb{B}, \norm{\cdot})$ and a compact subset $\mathcal{T} \subset \mathbb{B}^{\mathcal{I}_{n,d}}$. For $x \in \R^n$, we define
\begin{align}		\label{eqn:fTauNew}
	f(x) \coloneqq f_{\mathcal{T}}(x) \coloneqq \sup_{t \in \mathcal{T}}  \Big \lVert \sum_{I \in \mathcal{I}_{n,d}} x_I t_I \Big\rVert,
\end{align} 
where $x_I \coloneqq \prod_{i \in I} x_i$. Write $\nabla^{(j)} f_t(x)$ for the $k$-tensor of all partial derivatives of order $j$. For any $j \in \{1, \ldots, d\}$, we set
\[
W_j = W_j(x) = \sup_{t \in \mathcal{T}} \abs{\nabla^{(j)} f_t(x)}_{\mathrm{op}}.
\]
Moreover, write
\[
\tilde{W}_j = \tilde{W}_j(x) = \lvert (\sup_{t \in \mathcal{T}} \lVert \partial_{i_1 \ldots i_j} f_t(x) \rVert)_{i_1, \ldots, i_j} \rvert_\mathrm{op}.
\]
Here, one first takes the supremum in $\mathcal{T}$ of the Banach space norm of every entry of $\nabla^{(j)} f_t$ and only then the operator norm. Obviously, $W_j \le \tilde{W}_j$ for all $j \in \{1, \ldots, d\}$.

Concentration properties for functionals as in \eqref{eqn:fTauNew} have been studied for independent Rademacher variables $X_1, \ldots, X_n$ (i.\,e. $\IP(X_i = +1) = \IP(X_i = -1) = 1/2$) and $\mathbb{B} = \mathbb{R}$ in \cite[Th.\ 14]{BBLM05} for all $d \ge 2$, and under certain technical assumptions in \cite{Ada15}. In \cite{GSS21a}, the following result was shown.

\begin{satz}			\label{theorem:ChaosInIndependentOrPartialLSI}
	Let $\mu$ be a probability measure with support in $[a,b]^n$ for some real numbers $a < b$ which satisfies a $\dpartial\mathrm{-LSI}(\sigma^2)$. For $f = f(x)$ as in \eqref{eqn:fTauNew} and all $r \ge 2$ we have
	\begin{align*}
		\norm{(f - \IE f)_+}_r &\le \sum_{j = 1}^d (2\sigma^2 (b-a)^2 (r-3/2))^{j/2} \IE W_j, \\
		\norm{f - \IE f}_r &\le \sum_{j = 1}^d (2(b-a)^2 r)^{j/2} \IE\tilde{W}_j.
	\end{align*}
	Consequently, for any $t \ge 0$
	\begin{align*}		\label{eqn:ConcentrationEstimateNumber1}
			\mu \left( f - \IE f \ge t \right) &\le 2 \exp \Big( - \frac{1}{2\sigma^2(b-a)^2} \min_{j=1,\ldots,d} \Big( \frac{t}{de \IE W_j} \Big)^{2/j} \Big) \\
			&\le 2 \exp \Big( -\frac{1}{2e^2\sigma^2(b-a)^2d^2} \min_{j=1,\ldots,d} \Big( \frac{t}{\IE W_j} \Big)^{2/j}  \Big),
	\end{align*}
	and the same inequalities hold for $\mu(|f-\E f| \ge t)$ with $\IE W_j$ replaced by $\IE \tilde{W}_j$.
\end{satz}

We will not provide a complete proof but just point out that here, the inequalities $\abs{\mathfrak{h}^+ W_j} \le (b-a)W_{j+1}$ as well as $\abs{\mathfrak{h}^+ W_j} \le (b-a)W_{j+1}$ (formally setting $W_0 = \tilde{W}_0 = f$) are checked. In particular, the $W_k$ may be regarded as higher order difference operators in this context according to the framework addressed in Assumptions \ref{ass:Set2} and Proposition \ref{prop:HigherOrder3}.

\subsection{Poisson functionals}

Let us briefly discuss one final example, which unlike the previous ones is not a direct application of Theorem \ref{thm:HigherOrder2}. However, it is closely related to the proof strategy of the results for independent random variables due to \cite{BBLM05} and raises some interesting open questions.

Let $(\mathcal{X}, \mathcal{B})$ be a measurable space supplied with a $\sigma$-finite measure $\mu$. By $\sfN(\mathcal{X})$ we denote the space of $\sigma$-finite counting measures on $\mathcal{X}$. The $\sigma$-field $\cN(\mathcal{X})$ is defined as the smallest $\sigma$-field on $\sfN(\mathcal{X})$ such that the evaluation mappings $\xi\mapsto\xi(B)$, $B\in\mathcal{B}$, $\xi\in\sfN(\mathcal{X})$ are measurable. A point process on $\mathcal{X}$ is a measurable mapping with values in $\sfN(\mathcal{X})$ defined over some fixed probability space $(\Omega,\mathcal{A},\P)$. By a Poisson process $\eta$ on $\mathcal{X}$ with intensity measure $\mu$ we understand a point process with the following two properties:
\begin{itemize}
    \item[(i)] for any $B\in\mathcal{B}$ the random variable $\eta(B)$ is Poisson distributed with mean $\mu(B)$;
    \item[(ii)] for any $n\in\N$ and pairwise disjoint sets $B_1,\ldots,B_n\in\mathcal{B}$ the random variables $\eta(B_1),\ldots,\eta(B_n)$ are independent.
\end{itemize}

A Poisson functional is a random variable $F$ $\PP$-almost surely satisfying $F=f(\eta)$ for some measurable $f:\sfN(\mathcal{X})\to\R$. In this case $f$ is called a representative of $F$. If $\PP_\eta$ denotes the distribution of the Poisson process $\eta$ we will write $L^r(\PP_\eta)$, $r> 0$, for the space of Poisson functionals $F$ satisfying $\E|F|^r<\infty$. For a Poisson functional $F$ with representative $f$ and $x\in\mathcal{X}$ we define the difference operator $D_xF$ by putting
$$
D_xF:=f(\eta+\delta_x)-f(\eta),
$$
where $\delta_x$ stands for the Dirac measure at $x$. For further background on Poisson processes on general state spaces we refer the reader to the monograph \cite{LP17}.

Recalling the notion of $\Phi$-entropies and the functions $\Phi_r$, the following modified $\Phi_q$-Sobolev inequality is a specialization of Chafa\"{i} \cite[Sec.\ 5.1]{Cha04}.

\begin{proposition}\label{prop:PhiRSovolev}
Fix $q\in(1,2)$ and let $F\in L^1(\PP_\eta)$ be a Poisson functional satisfying $F\geq 0$ $\PP$-almost surely. Then
$$
\Ent_{\Phi_q}(F) \leq \E\Big[\int_\mathcal{X}\big(D_xF^{2\over q}-{2\over q}F^{{2\over q}-1}\,(D_xF)\big)\,\mu(d x)\Big].
$$
\end{proposition}

As $q \to 1$, the modified $\Phi_q$-Sobolev inequality turns into the $L^1$-version of the Poincar\'e inequality for Poisson functionals, while as $q\to 2$, we get back Wu's modified log-Sobolev inequality \cite[Th.\ 1.1]{Wu00}.

By arguments modelled upon \cite{BBLM05}, we can now derive moment bounds for Poisson functionals. Apart from the quantity $\kappa$ known from Theorem \ref{th:BBLM}, we shall need the quantities
\begin{align*}
V^+ &:= \int_\XX (D_x F)_-^2 \,\mu (d x) + \int_\XX (F(\eta) - F(\eta - \delta_x))_+^2 \,\eta(d x),\\
    V^- &:= \int_\XX (D_x F)_+^2\, \mu (d x) + \int_\XX (F(\eta) - F(\eta - \delta_x))_-^2 \,\eta(d x),\\
    V &:= \int_\XX (D_x F)^2 \,\mu (d x) + \int_\XX (F(\eta) - F(\eta - \delta_x))^2 \,\eta(d x).
\end{align*}

\begin{satz}[\cite{GST21}, Th.\ 4.1]\label{prop:ConcentrationLpBounds}
Let $r\geq 2$ and $F\in L^{1}(\PP_\eta)$. Then
\begin{align*}
\lVert (F-\E[F])_+ \rVert_r &\le \sqrt{2\kappa r \lVert V^+ \rVert_{r/2}} = \sqrt{2\kappa r} \,\lVert \sqrt{V^+} \rVert_r,\\
\lVert (F-\E[F])_- \rVert_r &\le \sqrt{2\kappa r \lVert V^- \rVert_{r/2}} = \sqrt{2\kappa r} \,\lVert \sqrt{V^-} \rVert_r,\\
\lVert F-\E[F] \rVert_r &\le \sqrt{8\kappa r \lVert V \rVert_{r/2}} = \sqrt{8\kappa r}\, \lVert \sqrt{V} \rVert_r.
\end{align*}
\end{satz}

Clearly, $\sqrt{V^\pm}$ and $\sqrt{V}$ can be regarded as difference operators, and in this sense, one may argue that Proposition \ref{prop:ConcentrationLpBounds} establishes Assumptions \ref{ass:Set1} (1)\&(2). By standard arguments, one can now derive the following concentration bounds for Poisson functionals, which recover Prop.\ 1.2, Cor.\ 3.3 (ii) and Cor.\ 3.4 (ii) from \cite{BP16} (up to absolute constants).

\begin{korollar}[\cite{GST21}, Cor.\ 4.3]\label{cortb}
Let $F\in L^1(\PP_\eta)$. If $\PP$-almost surely $V^+ \le L$ or $V^- \le L$ for some $L > 0$, we have that
\[
\PP (F - \E[F] \ge t) \le e^{-ct^2/L}\qquad \text{or}\qquad \PP (F - \E[F] \le -t) \le e^{-ct^2/L}
\]
for every $t \ge 4\sqrt{\kappa}$ and $c = \log(2)/(8\kappa)$, respectively. Moreover, if $\PP$-almost surely $V \le L$ for some $L > 0$, we have that
\[
\PP (|F - \E[F]| \ge t) \le 2e^{-ct^2/L}
\]
for every $t \ge 0$ and $c = \log(2)/(16\kappa)$.
\end{korollar}

Note that in the language of higher order concentration, Corollary \ref{cortb} yields first order concentration results. In fact, a related approach to concentration inequalities on Poisson spaces, also building upon \cite{BBLM05} but using Beckner inequalities, has been performed in \cite{APS22} in a much broader context, cf.\ also the corresponding comments and remarks in \cite{GST21}.

\begin{bemerkung}[Boundedness of $F$ vs.\ boundedness of $V^+$]\label{rm:BoundVF}
In \cite[Rem.\ 4.5]{GST21}, it was demonstrated that boundedness of $V^+$ and boundedness of $F$ do not coincide. More precisely, examples were given which showed that neither implication is true.

It seems that this is part of a general lack of understanding of these quantities. Clearly, $\sqrt{V}$ and $\sqrt{V^\pm}$ could be regarded as difference operators in the sense of \eqref{eq:DiffOp}. However, it is not clear whether there is a class of ``natural'' / first-order type functionals (like linear functions for the usual gradient on $\R^n$) which have bounded ``differences'' $\sqrt{V}$ or $\sqrt{V^\pm}$.

Likewise, in the context of the present survey, one could ask for higher order concentration results for Poisson functionals, thus working with Assumptions \ref{ass:Set1} or Assumptions \ref{ass:Set2}. However, once again we lack an understanding or even a feeling of which functionals we may address this way. Continuing the questions raised above, one may thus ask for an analogue of ``quadratic forms'' (etc.) for the Poisson space.
\end{bemerkung}


\begin{thebibliography}{99}

	\bibitem[Ada06]{Ada06} Adamczak, R.: \emph{Moment inequalities for U-statistics}, Ann. Probab. \emph{34}(6) (2006), 2288--2314.
	
	\bibitem[Ada15]{Ada15} Adamczak, R.: \emph{A note on the {H}anson-{W}right inequality for random vectors with dependencies}, Electron. Commun. Probab. \textbf{20} (2015), no. 72, 13 pp.
	
	\bibitem[ABW17]{ABW17} Adamczak, R., Bednorz, W., and Wolff, P.: \emph{Moment estimates implies by modified log-Sobolev inequalities}, ESAIM: P\&S \textbf{21} (2017), 467--494.
	
	\bibitem[ALM18]{ALM18} Adamczak, R., Lata{\l}a, R., and Meller, R.: \emph{Hanson-{W}right inequality in {B}anach spaces}, Ann. Inst. Henri Poincar\'{e} Probab. Stat. \textbf{56}(4) (2020), 2356--2376.
	
	\bibitem[APS22]{APS22} Adamczak, R., Polaczyk, B., Strzelecki, M.: \emph{Modified log-Sobolev inequalities, Beckner inequalities and moment estimates}, J. Funct. Anal. \textbf{282}(7) (2022), Article 109349, 76 pp.
	
	\bibitem[AW15]{AW15} Adamczak, R., and Wolff, P.: \emph{Concentration inequalities for non-Lipschitz functions with bounded derivatives of higher order}, Probab. Theory Relat. Fields \textbf{162}(3--4) (2015), 531--586.
	
	\bibitem[AS94]{AS94} Aida, S., and Stroock, D.: \emph{Moment estimates derived from Poincar\'{e} and logarithmic Sobolev inequalities}, Math. Res. Lett. \textbf{1}(1) (1994), 75--86.
	
	\bibitem[AG93]{AG93} Arcones, M.\, A., and Gin\'{e}, E: \emph{On decoupling, series expansions, and tail behavior of chaos processes}, J. Theoret. Probab. \textbf{6}(1) (1993), 101--122.
	
	\bibitem[BP16]{BP16} Bachmann, S., and Peccati, G.: \emph{Concentration bounds for geometric Poisson functionals: Logarithmic Sobolev inequalities revisited}, Electron. J. Probab. \textbf{21} (2016), No. 6, 44 pp.
	
	\bibitem[BZA24]{BZA24} Bendokat, T., Zimmermann, R., and Absil, P.-A.: \emph{A Grassmann manifold handbook: basic geometry and computational aspects}, Adv. Comput. Math. \textbf{50} (2024), No. 6, 51 pp.
	
	\bibitem[Bob10]{Bob10} Bobkov, S.\,G.: \emph{The growth of $L^p$-norms in presence of logarithmic Sobolev inequalities}, Vestnik Syktyvkar Univ. \textbf{11}(2) (2010), 92--111.
	
	\bibitem[BCG17]{BCG17} Bobkov, S.\,G., Chistyakov, G.\,P., and G\"{o}tze, F.: \emph{Second-order concentration on the sphere}, Commun. Contemp. Math. \textbf{19}(5) (2017), Article 1650058.
	
	\bibitem[BCG23]{BCG23} Bobkov, S.\,G., Chistyakov, G.\,P., and G\"{o}tze, F.: \emph{Concentration and Gaussian Approximation for Randomized Sums}, Probability Theory and Stochastic Modelling, Vol. 104, Springer, Cham, 2023, xvii+484 pp.
	
	\bibitem[BG99]{BG99} Bobkov, S.\, G., and G{\"o}tze, F.: \emph{Exponential integrability and transportation cost related to logarithmic {S}obolev inequalities}, J. Funct. Anal. \textbf{163}(1) (1999), 1--28.
	
	\bibitem[BGS19]{BGS19} Bobkov, S.\,G., G\"{o}tze, F., and Sambale, H.: \emph{Higher order concentration of measure}, Commun. Contemp. Math. \textbf{21}(3) (2019), Article 1850043.
%
	\bibitem[BT06]{BT06} Bobkov, S.\,G., and Tetali, P.: \emph{Modified logarithmic {S}obolev inequalities in
	discrete settings}, J. Theoret. Probab. \textbf{19}(2) (2006), 289--336.
	
	\bibitem[BZ05]{BZ05} Bobkov, S.\,G., and Zegarlinski, B.: \emph{Entropy Bounds and Isoperimetry}, Mem. Amer. Math. Soc. \textbf{176} (2005), no. 829, x+69 pp.
	
	\bibitem[Bon68]{Bon68} Bonami, A.: \emph{Ensembles $\Lambda(p)$ dans le dual de $D^{\infty }$}, Ann. Inst. Fourier (Grenoble) 1968 \textbf{18}(2) (1969), 193--204.
	
	\bibitem[Bon70]{Bon70} Bonami, A.: \emph{\'{E}tude des coefficients de Fourier des fonctions de $L^{p}(G)$}, Ann. Inst. Fourier (Grenoble) 1970 \textbf{20}(2) (1971), 335--402.
	
	\bibitem[Bor84]{Bor84} Borell, C.: \emph{On the Taylor series of a Wiener polynomial}, in: Seminar Notes on multiple stochastic integration, polynomial chaos and their integration, Case Western Reserve University, Cleveland, 1984. 
	
	\bibitem[BBLM05]{BBLM05} Boucheron, S., Bousquet, O., Lugosi, G., Massart, P.: \emph{Moment inequalities for functions of independent random variables}. Ann. Probab. 33(2) (2005), 514--560.
	
	\bibitem[BLM13]{BLM13} Boucheron, S., Lugosi, G., and Massart, P.: \emph{Concentration inequalities. A nonasymptotic theory of independence}, Oxford University Press, Oxford (2013).
	
	\bibitem[Cha04]{Cha04} Chafaï, D.: \emph{Entropies, convexity, and functional inequalities: on $\Phi$-entropies and $\Phi$-Sobolev inequalities}, J. Math. Kyoto Univ. \textbf{44}(2) (2004), 325--363.
	
	\bibitem[DSC96]{DSC96} Diaconis, P., and Saloff-Coste, L.: \emph{Logarithmic Sobolev inequalities for finite Markov chains}, Ann. Appl. Probab. \textbf{6}(3) (1996), 695--750.
	
	\bibitem[EAS98]{EAS98} Edelman, A., Arias, T.\,A., and Smith, S.\,T.: \emph{The geometry of algorithms with orthogonality constraints}, SIAM J. Matrix Anal. Appl. \textbf{20}(2) (1998), 303--353.
	
	\bibitem[GS19]{GS19} G\"{o}tze, F., and Sambale, H.: \emph{Higher order concentration in presence of Poincaré-type inequalities}, High Dimensional Probability VIII, Eds. N. Gozlan, R. Lata{\l}a, K. Lounici and M. Madiman, Birkh\"{a}user, Springer, Progress in Probability \textbf{74} (2019), 55--69.
	
	\bibitem[GS23]{GS23} G\"{o}tze, F., and Sambale, H.: \emph{Higher order concentration on Stiefel and Grassmann manifolds}, Electron. J. Probab. \textbf{28} (2023), no. 79, 30 pp.
	
	\bibitem[GS24]{GS24} G\"{o}tze, F., and Sambale, H.: \emph{Concentration of measure on spheres and related manifolds}, arXiv:2408.04346.
	
	\bibitem[GSS19]{GSS19} G\"{o}tze, F., Sambale, H., and Sinulis, A.: \emph{Higher order concentration for functions of weakly dependent random variables}, Electron. J. Probab. \textbf{24} (2019), no. 85, 19 pp.
	
	\bibitem[GSS21a]{GSS21a} G\"{o}tze, F., Sambale, H., and Sinulis, A.: \emph{Concentration Inequalities for Bounded Functionals via Log-Sobolev-Type Inequalities}, J. Theor. Probab. \textbf{34} (2021), 1623--1652.
	
	\bibitem[GSS21b]{GSS21b} G\"{o}tze, F., Sambale, H., and Sinulis, A.: \emph{Concentration inequalities for polynomials in $\alpha$-sub-exponential random variables}, Electron. J. Probab. \textbf{26} (2021), No. 48, 22 pp.
		
	\bibitem[Gro75]{Gro75} Gross, L.: \emph{Logarithmic Sobolev inequalities}, Amer. J. Math. \textbf{97}(4) (1975), 1061--1083.
	
	\bibitem[GST21]{GST21} Gusakova, A., Sambale, H., and Thäle, C.: \emph{Concentration on Poisson spaces via modified $\Psi$-Sobolev inequalities}, Stoch. Proc. Appl. \emph{140} (2021), 216--235.
	
	\bibitem[HW71]{HW71} David L. Hanson and Farroll T. Wright. ``A bound on tail probabilities for quadratic forms in independent random variables''. {\em Ann. Math. Statist.} \textbf{42} (1971), 1079--1083.
	
	\bibitem[HKZ12]{HKZ12} Hsu, D., Kakade, S.\, M., and Zhang, T.: \emph{A tail inequality for quadratic forms of subgaussian random vectors}, Electron. Commun. Probab. \textbf{17} (2012), no. 52, 6 pp.
	
	\bibitem[KPT20]{KPT20} Kabluchko, Z., Prochno, J., and Th\"{a}le, C.: \emph{A new look at random projections of the cube and general product measures}, Bernoulli \textbf{27}(3) (2021), 2117--2138.
	
	\bibitem[KV00]{KV00} Kim, J.\, H., and Vu, V.\, H.: \emph{Concentration of multivariate polynomials and its applications}, Combinatorica \textbf{20}(3) (2000), 417--434.
	
	\bibitem[KLS21]{KLS21} Kn\"{o}pfel, H., L\"{o}we, M., and Sambale, H.: \emph{Large deviations, a phase transition, and logarithmic Sobolev inequalities in the block spin Potts model}, Electron. Commun. Probab. \textbf{26} (2021), no. 29, 14 pp.
	
	\bibitem[KL15]{KL15} Kolesko, K., and Lata{\l}a, R.: \emph{Moment estimates for chaoses generated by symmetric random variables with logarithmically convex tails}, Statist. Probab. Lett. \textbf{107} (2015), 210--214.
	
	\bibitem[LP17]{LP17} Last, G., and Penrose, M.\, D.: \emph{Lectures on the Poisson Process}, Cambridge University Press, 2017.
	
	\bibitem[Lat06]{Lat06} Lata{\l}a, R.: \emph{Estimates of moments and tails of Gaussian chaoses}, Ann. Probab. \textbf{34}(6) (2006), 2315--2331.
	
	\bibitem[LO00]{LO00} Lata{\l}a, R., and Oleszkiewicz, K.: \emph{Between Sobolev and Poincar\'e}, Geometric aspects of functional analysis, 147--168, Lecture Notes in Math., 1745, Springer, Berlin, 2000.
	
	\bibitem[Led96]{Led96} Ledoux, M.: \emph{On Talagrand's deviation inequalities for product measures}, ESAIM Prob. \& Stat. \textbf{1} (1996), 63--87.
	
	\bibitem[Led01]{Led01} Ledoux, M.: \emph{The concentration of measure phenomenon}, Mathematical Surveys and Monographs, vol. 89. American Mathematical Society, Providence, RI (2001).
	
	\bibitem[Mar19]{Mar19} Marton, K.: \emph{Logarithmic Sobolev inequalities in discrete product spaces}, Combin. Probab. Comput. \textbf{28}(6) (2019), 919--935.
	
	\bibitem[Mec19]{Mec19} Meckes, E.: \emph{The Random Matrix Theory of the Classical Compact Groups}, Cambridge University Press, 2019.
	
	\bibitem[Mil71]{Mil71} Milman, V.\, D.: \emph{A new proof of A. Dvoretzky's theorem on cross-sections of convex bodies}, Funkcional. Anal. i Prilo\v{z}en. \textbf{5}(4) (1971), 28--37.
	
	\bibitem[MS86]{MS86} Milman, V.\, D., and Schechtman, G.: \emph{Asymptotic theory of finite-dimensional normed spaces}, Vol. 1200. Springer-Verlag, Berlin, 1986, viii+156 pp.
	
	\bibitem[MW82]{MW82} Mueller, C.\,E., and Weissler, F.\,B.: \emph{Hypercontractivity for the heat semigroup for ultraspherical polynomials and on the $n$-sphere}, J. Funct. Anal. \textbf{48} (1992), 252--283.
	
	\bibitem[Nel73]{Nel73} Nelson, E.: \emph{The free Markoff field}, J. Functional Analysis \textbf{12} (1973), 211--227.
	
	\bibitem[PTT19]{PTT19} Prochno, J., Th\"{a}le, C., and Turchi, N.: \emph{Geometry of $\ell_p^n$-balls: Classical results and recent developments}, High Dimensional Probability VIII, Eds. N. Gozlan, R. Lata{\l}a, K. Lounici and M. Madiman, Birkh\"{a}user, Springer, Progress in Probability \textbf{74} (2019), 131--150.
	
	\bibitem[RR91]{RR91} Rachev, S., and R\"{u}schendorf, L.: \emph{Approximate independence of distributions on spheres and their stability properties}, Ann. Probab. \textbf{19}(3) (1991), 1311--1337.
	
	\bibitem[RS14]{RS14} Raginsky, M., and Sason, I.: \emph{Concentration of measure inequalities in information theory, communications, and coding}, Now Publishers Inc., 2014, 1--260.
	
	\bibitem[RV13]{RV13} Rudelson, M., and Vershynin, R.: \emph{Hanson-Wright inequality and sub-Gaussian concentration}, Electron. Commun. Probab. \textbf{18} (2013), no. 82, 9 pp.
	
	\bibitem[Sam23]{Sam23} Sambale, H.: \emph{Some notes on Concentration for $\alpha$-subexponential random variables}, in: High Dimensional Probability IX -- the Ethereal Volume. Progr. Probab., Vol. 80, Eds. R. Adamczak, N. Gozlan, K. Lounici, and M. Madiman, Birkh\"{a}user/Springer, Cham, 2023, 167--192.
	
	\bibitem[SS20]{SS20} Sambale, H., and Sinulis, A.: \emph{Logarithmic Sobolev inequalities for finite spin systems and applications}, Bernoulli \textbf{26}(3) (2020), 1863--1890.
	
	\bibitem[SS21]{SS21} Sambale, H., and Sinulis, A.: \emph{Modified log-{S}obolev inequalities and two-level concentration}, ALEA Lat. Am. J. Probab. Math. Stat. \textbf{18} (2021), 855--885.
	
	\bibitem[Sat23]{Sat23} Sato, H.: \emph{Riemannian optimization on unit sphere with $p$-norm and its applications}, Comput. Optim. Appl. \textbf{85} (2023), 897--935.
	
	\bibitem[SZ90]{SZ90} Schechtman, G., and Zinn, J.: \emph{On the volume of the intersection of two $L_p^n$ balls}, Proc. Amer. Math. Soc. \textbf{110}(1) (1990), 217--224.
	
	\bibitem[Sud78]{Sud78} Sudakov, V.\,N.: \emph{Typical distributions of linear functionals in finite-dimensional spaces of higher dimension},  Dokl. Akad. Nauk SSSR \textbf{243} (1978), 1402--1405.
	
	\bibitem[vHa16]{vHa16} Handel, R.: \emph{Probability in High Dimension}, 2016.
	
	\bibitem[VW15]{VW15} Vu, V.\, H., and Wang, K.: \emph{Random weighted projections, random quadratic forms and random eigenvectors}, Random Structures Algorithms \emph{47}(4) (2015), 792--821.
	
	\bibitem[Wol13]{Wol13} Wolff, P.: \emph{On some Gaussian concentration inequality for non-Lipschitz functions}, in: High dimensional probability VI, vol. 66, Birkhäuser/Springer, Basel, 2013, 103--110.
	
	\bibitem[Wri73]{Wri73} Wright, F.\, T.: \emph{A bound on tail probabilities for quadratic forms in independent random variables whose distributions are not necessarily symmetric}, Ann. Probab. \textbf{1}(6) (1973), 1068--1070.
	
	\bibitem[Wu00]{Wu00} Wu, L.: \emph{A new modified logarithmic Sobolev inequality for Poisson point processes and several applications}, Probab. Theory Rel. Fields \textbf{118} (2000), 427--438.
	
\end{thebibliography}
\end{document}